\documentclass{amsart}

\usepackage{amsmath,amssymb,mathrsfs,amscd,amsthm}
\usepackage{braket}
\usepackage[dvips]{graphicx} 
\usepackage{psfrag} 
\usepackage{mathtools} 
\usepackage{esint} 
\usepackage{nccbbb} 
\usepackage[utf8]{inputenc} 
\usepackage{hyperref}
\usepackage{tikz}

\newtheorem{theorem}{Theorem}[section]
\newtheorem{lemma}[theorem]{Lemma}
\newtheorem{proposition}[theorem]{Proposition}
\newtheorem{corollary}[theorem]{Corollary}
\newtheorem{definition}[theorem]{Definition}
\theoremstyle{definition}
\newtheorem{remark}[theorem]{Remark}

\numberwithin{equation}{section}
\allowdisplaybreaks

\newcommand{\BK}[1]{ {\left( #1 \right)} }
\newcommand{\sqBK}[1]{ {\left[ #1 \right]} }
\newcommand{\curBK}[1]{ {\left\{ #1 \right\}} }
\newcommand{\absBK}[1]{ {\left| #1 \right|} }
\newcommand{\VertBK}[1]{ {\Vert #1 \Vert} }
\newcommand{\angleBK}[1]{ {\left < #1 \right >} }
\newcommand{\ceilBK}[1]{ {\lceil #1 \rceil} }
\newcommand{\floorBK}[1]{ {\lfloor #1 \rfloor} }
\newcommand{\dotBK}{ {(\,\cdot\,)} }

\usepackage{acronym}

\acrodef{KPZ}[\textsc{kpz}]{Kardar--Parisi--Zhang}
\acrodef{SHE}[\textsc{she}]{stochastic heat equation}
\acrodef{dSHE}[\textsc{dshe}]{discrete stochastic heat equation}

\newcommand{\Omegali}{ \Omega^\text{lin} }
\newcommand{\Omegaqd}{ \Omega^\text{qd} }

\newcommand{\E}{\varepsilon}

\newcommand{\lambdaE}{{\lambda\E^{1/2}}}
\newcommand{\nuE}{\nu_\varepsilon}
\newcommand{\gammaE}{\gamma^\varepsilon}
\newcommand{\rhoE}{\rho^\varepsilon}
\newcommand{\sigmaE}{\sigma^\varepsilon}
\newcommand{\phiE}{\phi^\varepsilon}
\newcommand{\psiE}{\psi^\varepsilon}

\newcommand{\cE}{c^\varepsilon}
\newcommand{\fE}{f^\varepsilon}
\newcommand{\gE}{g^\varepsilon}
\newcommand{\qE}{q^\varepsilon}
\newcommand{\alphaE}{\alpha^\varepsilon}

\newcommand{\DE}{D^\varepsilon}
\newcommand{\FE}{F^\varepsilon}
\newcommand{\HE}{H^\varepsilon}
\newcommand{\KE}{K^\varepsilon}
\newcommand{\LE}{L^\varepsilon}
\newcommand{\ME}{M^\varepsilon}
\newcommand{\NE}{N^\varepsilon}
\newcommand{\RE}{R^\varepsilon}
\newcommand{\UE}{U^\varepsilon}
\newcommand{\UEd}{U^{\varepsilon,\delta}}
\newcommand{\ZE}{Z^\varepsilon}
\newcommand{\LambdaE}{\Lambda^\varepsilon}
\newcommand{\OmegaE}{\Omega^\varepsilon}

\newcommand{\calZE}{\mathcal{Z}^\varepsilon}
\newcommand{\scrZE}{\mathscr{Z}^\varepsilon}

\newcommand{\tilr}{\widetilde{r}}
\newcommand{\tilf}{\widetilde{f}}
\newcommand{\tilg}{\widetilde{g}}
\newcommand{\tilcalC}{\widetilde{\mathcal{C}}}

\newcommand{\tilD}{\widetilde{D}}
\newcommand{\tilS}{\widetilde{S}}
\newcommand{\tilT}{\widetilde{T}}

\newcommand{\tilDelta}{\widetilde{\Delta}}
\newcommand{\tilPsi}{\widetilde{\Psi}}

\newcommand{\tilrE}{{\widetilde{r}^\varepsilon}}

\newcommand{\tilNE}{\widetilde{N}^\varepsilon}
\newcommand{\tilLambdaE}{\widetilde{\Lambda}^\varepsilon}
\newcommand{\tilUEd}{\widetilde{U}^{\varepsilon,\delta}}

\newcommand{\fNTE}{f^{N_\varepsilon}_{T,\varepsilon}}
\newcommand{\barflTE}{ \bar{f}^{l}_{T,\varepsilon}}

\newcommand{\fntE}{f^{n}_{t,\varepsilon}}
\newcommand{\fnmtE}{f^{n+m}_{t,\varepsilon}}

\newcommand{\deltaE}{{\delta\varepsilon^{-\frac12}}}
\newcommand{\sE}{{s,\varepsilon}}
\newcommand{\tE}{{t,\varepsilon}}
\newcommand{\TE}{{T,\varepsilon}}

\newcommand{\barf}{\bar f}

\newcommand{\bari}{{\bar i}}
\newcommand{\barj}{{\bar j}}
\newcommand{\bark}{{\bar k}}

\newcommand{\bbC}{ \mathbb{C} }
\newcommand{\bbE}{ \mathbb{E} }
\newcommand{\bbP}{ \mathbb{P} }
\newcommand{\bbN}{ \mathbb{N} }
\newcommand{\bbR}{ \mathbb{R} }
\newcommand{\bbZ}{ \mathbb{Z} }

\newcommand{\calA}{ \mathcal{A} }
\newcommand{\calB}{ \mathcal{B} }
\newcommand{\calC}{ \mathcal{C} }
\newcommand{\calD}{ \mathcal{D} }
\newcommand{\calE}{ \mathcal{E} }

\newcommand{\calG}{ \mathcal{G} }
\newcommand{\calI}{ \mathcal{I} }
\newcommand{\calL}{ \mathcal{L} }
\newcommand{\calQ}{ \mathcal{Q} }
\newcommand{\calT}{ \mathcal{T} }
\newcommand{\calZ}{ \mathcal{Z} }

\newcommand{\scrG}{ \mathscr{G} }
\newcommand{\scrF}{ \mathscr{F} }

\newcommand{\bfL}{ \mathbf{L} }
\newcommand{\bfp}{ \mathbf{p}^\varepsilon }

\newcommand{\fkT}{ \mathfrak{T} }
\newcommand{\fkf}{ \mathfrak{f} }

\DeclareMathOperator{\sign}{sign}
\DeclareMathOperator*{\avsum}{\mathrlap{\hspace{.2ex}\rotatebox[origin=c]{20}{\bf---}}{\sum}}
\DeclareMathOperator*{\avvsum}{\mathrlap{\hspace{.2ex}\rotatebox[origin=c]{20}{\bf--}}{\sum}}

\newcommand{\halfZ}{\mathbb{L}}
\newcommand{\Lap}{\bar\Delta}
\newcommand{\usZ}{\widehat{\calZ}}

\newcommand{\Dirilinf}{D^{\infty}_{l}}

\begin{document}

\title[Weakly Asymmetric Non-simple Exclusion Process]
{Weakly Asymmetric Non-simple Exclusion Process and the Kardar--Parisi--Zhang equation}

\author[A.\ Dembo]{$^*$Amir Dembo}
\address{$^{*\dagger}$Department of Mathematics, Stanford University, California 94305}
\author[L.-C.\ Tsai]{$^\dagger$Li-Cheng Tsai}

\date{\today}



\thanks{$^*$ Research partially supported by NSF grant DMS-1106627.}

\begin{abstract}
We analyze a class of non-simple exclusion processes
and the corresponding growth models
by generalizing G\"{a}rtner's discrete Cole--Hopf transformation.
We identify the main non-linearity and eliminate it 
by imposing a gradient type condition.
For hopping range at most $3$, using the generalized transformation,
we prove the convergence of the exclusion process toward the \ac{KPZ} equation.
This is the first \emph{universality} result
concerning interacting particle systems in the context of \ac{KPZ} universality class.
While this class of exclusion processes are not explicitly solvable,
we obtain the exact one-point limiting distribution
for the step initial condition by using the previous result of \cite{amir11} and our convergence result.
\end{abstract}

\maketitle

\section{Introduction}
\label{sec:intro}

In this paper we study the asymptotic behavior 
of weakly asymmetric exclusion processes 
with finite hopping ranges.
A general exclusion process 
is an interacting particle system on 
the half integer lattice $\frac12+\bbZ:=\halfZ$ 
under the constraint that each site holds at most one particle,
\cite{liggett05}.
Let $m\in\bbN$ be the hopping range of the exclusion process,
let $q_k>0$, $k=\pm1,\ldots,\pm m$, be the hopping rate,
satisfying $\sum_{k=1}^m(q_k+q_{-k})=1$,
and let
\begin{equation}\label{eq:Possion}
\curBK{ Q^k_t(y): k=\pm 1,\ldots,\pm m, ~ y\in\halfZ }
\end{equation}
be mutually independent (in $k$ and $y$) Poisson clocks, 
with each $Q^k_t(y)$ having rate $q_k$.
When the clock $Q^k_t(y)$ rings, 
the particle at $y$ (if exists)
hops to $y+k$ if this destination is unoccupied,
otherwise it stays at $y$.

We associate a growth model with an exclusion process.
Let $\eta_t(y)$ be the occupation variable 
of the exclusion process (in the spin convention) 
at position $y\in\halfZ$ and time $t\in[0,\infty)$
\begin{equation}\label{eq:spin:convention}
	\eta_t(y) 
:=
	\left\{\begin{array}{c@{,}l}
		1	& 	\text{ when the site } y \text{ is occupied at time } t,	\\
		-1	&	\text{ otherwise,}
	\end{array}\right.
\end{equation}
and let $h_t(x)$ be the accumulated flux of particles at $x\in\bbZ$.
More precisely,
$h_t(0)$ is the net flow of particles through 
$x=0$ during the time interval $[0,t]$,
counting \emph{left} going particles as positive,
and  
\begin{align}
	\label{eq:h:defn}
	h_t(x) &:= 
	h_t(0) + 
	\left\{\begin{array}{c@{,}l}
		\sum_{0<y<x} \eta_t(y)	& \text{ when } x>0,\\
		-\sum_{x<y<0} \eta_t(y)	& \text{ when } x<0. 
	\end{array}\right.
\end{align}
Note that the discrete gradient of $h_t$ yields $\eta_t$,
that is,
\begin{align*}
	h_t \BK{y+\frac12} - h_t\BK{y-\frac12} 
	= \eta_t(y) \in \curBK{\pm 1},
\end{align*}
and therefore $h_t\dotBK$ represents the height of 
a surface consisting of broken lines with slope $\pm 1$,
whose evolution is described by following growth model.
For $k>0$, when the clock $Q^k_t(y)$ rings,
if the hop from $y$ to $y+k$ is allowed,
that is $(\eta_t(y),\eta_t(y+k))=(1,-1)$,
decrease $h_t(x)$ by 2 for all $x\in(y,y+k)\cap\bbZ$,
otherwise do nothing.
For $k<0$, when the clock $Q^k_t(y)$ rings,
if the hop from $y$ to $y+k$ is allowed,
increase $h_t(x)$ by 2 for all $x\in(y+k,y)\cap\bbZ$,
otherwise do nothing.
See Figure \ref{fig:multiple:corner:growth}.
For $m=1$ this is the simple exclusion process
and the corresponding corner growth model.

\begin{figure}
\centering
\psfrag{A}[c]{$\scriptstyle-\frac{11}{2}$}
\psfrag{B}[c]{$\scriptstyle-\frac{9}{2}$}
\psfrag{C}[c]{$\scriptstyle-\frac{7}{2}$}
\psfrag{D}[c]{$\scriptstyle-\frac{5}{2}$}
\psfrag{E}[c]{$\scriptstyle-\frac{3}{2}$}
\psfrag{F}[c]{$\scriptstyle-\frac{1}{2}$}
\psfrag{G}[c]{$\scriptstyle\frac{1}{2}$}
\psfrag{H}[c]{$\scriptstyle\frac{3}{2}$}
\psfrag{I}[c]{$\scriptstyle\frac{5}{2}$}
\psfrag{J}[c]{$\scriptstyle\frac{7}{2}$}
\psfrag{K}[c]{$\scriptstyle\frac{9}{2}$}
\psfrag{L}[c]{$\scriptstyle\frac{11}{2}$}
\includegraphics[width=0.7\textwidth]{{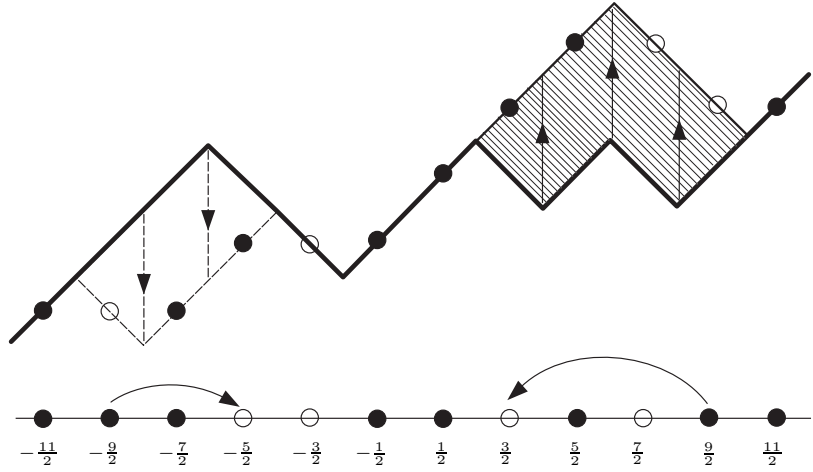}}
\caption{Non-simple exclusion process.
	The solid dots represent the particles,
	the hollow dots represent the empty sites,
	and the thick solid line represents the height $h_t\dotBK$.
	The horizontal arrows represent possible hops of the particles.
	The hop of two steps to the right decreases the height by 2 
	at the sites indicated by the dash lines,
	and the hop of three steps to the left increases the height by 2
	at the sites indicated by the shaded region. 
	}
\label{fig:multiple:corner:growth}
\end{figure}

We are interested in weakly asymmetric exclusion processes,
whereby the hopping rates depend on $\E$ in such a way that
\begin{equation}\label{eq:hop:rate:defn}
	\qE_k = \frac12 r_k \BK{ 1 - \gammaE_k \varepsilon^\frac12 },
\quad
	\qE_{-k} = \frac12 r_k \BK{ 1 + \gammaE_k \varepsilon^\frac12 },
\quad
	k=1,\ldots,m,
\end{equation} 
where $\E\to 0$ is a scaling parameter,
$r_1,\ldots,r_m$ are fixed numbers satisfying
\begin{equation}\label{eq:r:constraint}
	r_k > 0, \quad r_1+\ldots+r_m =1,
\end{equation}
and $\gammaE_k$ may depend on $\E$ 
in such a way that 
the limit $\lim_{\E\to 0}\gammaE_k$ exists and is finite,
for $k=1,\ldots,m$. 
For such exclusion processes, 
under the diffusive scaling 
\begin{align}\label{eq:diff:scaling}
	X=\E x, T=\E^2 t, \E\to 0
\end{align}
of space and time,
we expect the asymptotic fluctuation of $\E^\frac12 h_t$
to converge to a non-degenerated limit.
Indeed, Bertini and Giacomin \cite{bertini97} show that 
for simple ($m=1$) exclusion,
by choosing 
\begin{align}\label{eq:Gartner:para}
	r_1=1, 
\quad
	\gammaE_1 = 1 , 
\quad
	\nuE = \E^{-1} \BK{ 1 - (4\qE_1\qE_{-1})^\frac12 },
\end{align}
the scaled height function
\begin{align*}
	\HE_T(X) := \E^\frac12 h_{\E^{-2}T}(\E^{-1}X) - \nuE \E^{-1} T
\end{align*}
converges (as $\E\to 0$) to the solution of the \ac{KPZ} equation
\begin{equation}\label{eq:KPZ}
	\partial_T H
=
	\frac12 \partial^2_X H - \frac12 \BK{ \partial_X H }^2 + W,
	\quad
	H = H_T(X),
\end{equation}
where $W$ is the spacetime white noise:
$\bbE(W_T(X) W_S(X'))=\delta(T-S) \delta(X-X')$.
While one expect such a result to hold universally for a large collection of models,
the approach of \cite{bertini97} is restricted to the simple exclusion process.
In this paper,
we take the first step toward universality results by 
showing that the same convergence holds 
for non-simple exclusion processes
satisfying $m\leq 3$ and 
\begin{align}\label{eq:gamma:approx}
	\gammaE_k - \lambda 
	\BK{ \frac{2}{r_k} \sum_{k'=k+1}^m \frac{k'-k}{k} r_{k'} +1 } 
	= O(\E).
\end{align}
Hereafter $O(\E)$ stands for a generic function
satisfying $\sup_{\E\in(0,1)} |O(\E)\E^{-1}| < \infty$.

While the \ac{KPZ} equation \eqref{eq:KPZ} is a paradigm of 
equations describing randomly growing surfaces,
introduced by Kardar, Parisi, and Zhang \cite{kardar86},
it is mathematically ill-defined.
Indeed, generic solutions to stochastic differential equations
driven by the white noise exhibit non-differentiability.
Hence, $(\partial_X H)^2$ alone does not have a mathematical meaning.
Rather, the correct mathematical interpretation is the Cole--Hopf transformation to the \ac{SHE}
\begin{align}
	\label{eq:HCtrans}
	H_T(X) &= - \log Z_T(X),
\\
	\label{eq:SHE}
	\partial_T \calZ &= \frac12 \partial_X^2 \calZ + \calZ W.
\end{align}
Since the \ac{SHE} is linear, traditional stochastic calculus applies.
Further, formally $H_T(X)$ of \eqref{eq:HCtrans}
satisfies \eqref{eq:KPZ}, 
so we \emph{define} \eqref{eq:KPZ} by \eqref{eq:HCtrans} and \eqref{eq:SHE}.
This Cole--Hopf transformation dates back to \cite{huse85,kardar86}.
A more comprehensive approach of defining 
solutions to \eqref{eq:KPZ} can be found in \cite{hairer11}.
See \cite{assing11,goncalves10}
for recent development in defining \eqref{eq:KPZ}. 

Employing the non-rigorous renormalization argument of \cite{forster77},
Kardar, Parisi, and Zhang show that the scaling exponents of 
the fluctuation of $H$, space $X$, and time $T$ 
follow a $1:2:3$ ratio, 
signifying a new universality class---the \ac{KPZ} universality class.
This universality class describes 
various phenomena including paper wetting, 
crack formation, and burning fronts.
See \cite[Section 1.1.2]{corwin12} and the references therein.
Moreover, it connects various models describing other phenomenon
including directed last passage percolation, 
directed polymer in a random media, 
and polynuclear growth.
Recently, there has been intensive mathematical research
on instances of explicitly solvable models in this universality class
\cite{borodin11,borodin12,johansson00,johansson03,praehofer02}.
See also \cite{corwin12} and the references therein. 
They all confirm the $1:2:3$ scaling exponents
and have limiting one-point statistics related to random matrix theory,
for example the GUE or GOE Tracy--Widom distribution.
In this paper we provide 
the first instance of a collection of non-explicitly-solvable models
belonging the same universality class.

Specifically, 
in view of the Cole--Hopf transformation \eqref{eq:HCtrans},
for any fixed $\lambda >0$, define 
\begin{align}\label{eq:Z:defn}
	\ZE_t(x) := \exp \BK{ -\lambda \E^\frac12 h_t(x) + \E\nuE t },
\end{align}
for some deterministic $\nuE$ specified by \eqref{eq:nu:defn}.
We show here that under 
the diffusive scaling \eqref{eq:diff:scaling},
$\ZE_\cdot\dotBK$ converges to the mild solution of the \ac{SHE} \eqref{eq:SHE}.
Precisely,
we say a process $Z_\cdot\dotBK$ is a mild solution to the \ac{SHE} 
starting from the initial condition $Z_0\dotBK$ if
\begin{equation}\label{eq:SHE:mild}
	 Z_T(X)
=
	\int_\bbR P_{T}(X-X') Z_0(X') dX' 
	+ 
	\int_0^T \int_\bbR  
	 P_{T-S}(X-X') Z_S(X') W(dX'dS),
\end{equation}
where $P_T(X):=(2\pi T)^{-1/2}\exp(-X^2/2T)$ is the heat kernel,
and the stochastic integral is in the It\^o sense.
For the existence, uniqueness, continuity, and positivity of solutions 
to \eqref{eq:SHE:mild}, see \cite[Proposition 2.5]{corwin12}
and \cite{bertini95,bertini97,mueller91,walsh86}.
Let $\alpha$ denote the diffusivity of the symmetric part of the hopping rates, that is,
\begin{align}
	\label{eq:alpha:defn}
	\alpha &:= \sum_{k=1}^m k^2 r_k,
\end{align}
and extend $\ZE_t(x)$ of \eqref{eq:Z:defn} to $x\in\bbR$ by a linear interpolation.
Consider the scaled field 
\begin{align}\label{eq:calZE:defn}
	\calZE_{T}(X) := \ZE_{\E^{-2}\beta T}(\E^{-1}\beta'X),
\end{align}
where
\begin{align}\label{eq:beta:beta'}
	\beta:=\alpha^{-1}\lambda^{-4},
	\quad
	\beta':=\lambda^{-2}.
\end{align}
Let $\VertBK{f_t(x)}_l := [\bbE(|f_t(x)|^l)]^{1/l}$
denote the $L^l$-norm over randomness, $l\ge 1$.
Our main result is

\begin{theorem}\label{thm:main}
let $Z_T(X)$ be the mild solution to the \ac{SHE} starting from a $C(\bbR)$-valued process $Z_0(X)$,
and let $\calZE_T(X)$ be
defined by \eqref{eq:calZE:defn} and \eqref{eq:Z:defn} 
for a weakly asymmetric exclusion process
satisfying \eqref{eq:gamma:approx} and $m\leq 3$.
Suppose the initial condition $\calZE_0\dotBK$ satisfies
\begin{align}\label{eq:initial:cond:conv}
	\calZE_0\dotBK \text{ weakly converges to } Z_0\dotBK,
\end{align}
and for any $u\in(0,\frac12)$ 
there exist finite $C=C(u)$ and $a_0=a_0(u)$ such that
\begin{align}
	\label{eq:initial:cond:bd}
&
	\VertBK{\ZE_0(x)}_{14}
	\leq e^{a_0\E|x|} C,
	\\
	\label{eq:initial:cond:Holder}
&
	\VertBK{ \ZE_0(x)-\ZE_0(x')}_{14}
	\leq (\E|x-x'|)^u e^{a_0\E(|x|+|x'|)} C.
\end{align}

Then, under the Skorokhod topology of $D([0,\infty),C(\bbR))$,
the process $\calZE_\cdot\dotBK$ weakly converges to $Z_\cdot\dotBK$.
\end{theorem}

Bertini and Giacomin \cite{bertini97} proved Theorem \ref{thm:main} 
for the special case of $m=1$,
based on the work of G\"{a}rtner \cite{gaertner87}.
By choosing the parameters according to \eqref{eq:Gartner:para}
to match 3 non-degenerated identities,
G\"{a}rtner linearizes the drift term 
in the dynamical equation of $\ZE_t$,
making it a discrete Laplacian; thus $\ZE_t$ satisfies a \ac{dSHE}.
The difficulty in going beyond simple exclusion ($m>1$)
is that we encounter $2^m-1>3$ identities, 
so this type of reasoning fails.
Further, since we are at the fluctuation scale, 
where the time span is of $O(\E^{-2})$,
the nonlinearity of the drift causes much roughness,
and the contribution of the nonlinear terms 
in the drift are not even uniformly (in $\E$) bounded.

Instead of the approach of \cite{bertini97},
for $m>1$ we match the drift in the dynamical equation of $\ZE_t$ 
with a discrete Laplacian
\begin{align}\label{eq:disc:Laplace}
	\Lap := 2^{-1} \sum_{k=1}^m \tilrE_k (\Delta_k \ZE_t)(x)
\end{align}
slightly better than $O(\E^2)$,
where $\Delta_k f(x):= f(x+k)+f(x-k)-2f(x)$,
for some deterministic $\tilr_k\in\bbR$ specified by \eqref{eq:tilr:approx}.
To this end,
we first apply \eqref{eq:Z:defn} to get
\begin{align}\label{eq:Laplace:exp:expansion}
	\frac{\Lap}{\ZE_t(x)}
	=
	\frac12
	\sum_{k=1}^m
	\tilrE_k 
	\sqBK{ 
		\exp \Big( -\lambdaE\sum\limits_{y\in(x,x+k)}\eta_t(y) \Big)
		+ \exp \Big( \lambdaE\sum\limits_{y\in(x-k,x)}\eta_t(y) \Big)
		-2
		},
\end{align}
then Taylor-expand \eqref{eq:Laplace:exp:expansion},
and then match the $\eta$-linear and $\eta$-quadratic terms in \eqref{eq:Laplace:exp:expansion} 
with those in the drift up to $O(\E^2)$.
We find that such a matching requires that for each $j=1,\ldots,m$,
\begin{align}\label{eq:matching:eq}
\begin{split}
	\E^\frac12 \BK{ A^\E \tilrE }_j 
	&= 
	\E^\frac12 4^{-1}\sum_{k=j}^m r_k \BK{ u(\E) - \E \gammaE_k v(\E) } + O(\E^2),
\\
	\E \BK{ B \tilrE }_j 
	&= 
	\E 4^{-1}r_j  \BK{ \gammaE_j u(\E) - v(\E) } + O(\E^2),
\end{split}
\end{align}
where $A^\E$ and $B$ are the $m$-dimensional square matrices of entries
\begin{align}
	\label{eq:AE:defn}
	 A^\E_{jk} &:=
	 \bbbone_\curBK{j\le k}
	 \BK{
	 	\frac{\lambda}{2}  
		+ \E\lambda^3 \BK{ \frac{3k-2}{12} - \frac{j-1}{2} }
		},
\\
	\label{eq:B:defn}
	B_{jk}
	&:= \lambda^2 (k-j)_+ j^{-1},
\end{align}
and $u(\E)$ and $v(\E)$ are the analytic functions
\begin{align}\label{eq:uv:defn}
\begin{split}
	u(\E) &:= \E^{-\frac12} \sinh(2\lambda\E^\frac12),
	\quad
	v(\E) := \E^{-1}(\cosh(2\lambda\E^\frac12) -1),
\\
	u(0) &:= \lim_{\E\to 0} u(\E) = 2\lambda,
	\quad
	v(0) := \lim_{\E\to 0} v(\E) = 2\lambda^2.
\end{split}
\end{align}
By choosing $\gammaE_k$ to be of the specific forms \eqref{eq:gamma:approx},
we ensure the identities \eqref{eq:matching:eq} for suitable choices of $\tilrE_k$.
Namely, as shown in Appendix \ref{sec:app},
\begin{lemma}\label{lem:gamma:tilr}
There exists some explicit $\tilr^*_k\in\bbR$ 
such that any $\gammaE$ of the form \eqref{eq:gamma:approx}
and any $\tilrE$ of the form
\begin{align}\label{eq:tilr:approx}
	\tilrE_k = r_k + \E \tilr^*_k + O(\E^\frac32)
\end{align}
satisfy the equation \eqref{eq:matching:eq}.
\end{lemma}
\noindent

To further reduce the difference
between the drift and $\Lap$,
we need to control $\eta$-cubic terms.
The drift contains no $\eta$-cubic (or any higher order) term,
and the $\eta$-cubic terms in \eqref{eq:Laplace:exp:expansion}
already assume the gradient form
(see Definition \ref{defn:calE:gradient}),
which can be controlled by averaging over a relevant space section.
However, 
the matching of $\eta$-quadratic terms actually 
causes extra $\eta$-cubic terms,
which are of gradient form only if $m\leq 3$, see Remark \ref{rmk:m<=3},
hence our assumption $m\le 3$.
While our argument of converting terms into gradient terms 
requires the assumptions \eqref{eq:gamma:approx} and $m\le 3$,
we predict that Theorem \ref{thm:main} 
with the scaling constants \eqref{eq:calZE:defn} and \eqref{eq:beta:beta'}
holds even without either assumptions.
To the extent of our knowledge,
this is the first work studying the fluctuation of exclusion processes
by converting terms into gradient terms,
and no work has been done in this regard tackling non-gradient terms. 

Even under the assumption $m\leq 3$,
we are left with $\eta$-nonlinear terms 
with deterministic coefficients of $O(\E^2)$.
To conclude the matching of the drift and $\Lap$,
we need to show that these $\eta$-nonlinear terms
\emph{actually converge to zero} 
after being averaged over the relevant space and time section.
This we do in Lemma \ref{lem:pre:replace},
and such an estimate is also required for controlling the martingale term.
Indeed, \cite{bertini97} uses an ergodic type analysis
to control the quadratic variation of the martingale,
which applies only to a specific $\eta$-quadratic term that we find in $m=1$.
Since we encounter more general $\eta$-nonlinear terms,
we need to resort to a different approach,
adapting to our setting the sketch of \cite[Proposition 3.28]{quastel12} 
based on the one-block and two-blocks estimates 
for the simple exclusion process on $\bbZ/N\bbZ$ of \cite{kipnis89a}.

The deterministic constant $\nuE$ in \eqref{eq:Z:defn} balances
the constants coming from the microscopic dynamical equation and \eqref{eq:disc:Laplace}.
It turns out to be 
\begin{align}\label{eq:nu:defn}
	\nuE := 
	\sum_{k=1}^m 4^{-1} k r_k \BK{ \gammaE_k u(\E) - v(\E) }
	+	
	\lambda^2 \sum_{k=1}^m 
	\tilrE_k \sqBK{ k + \E (12)^{-1} (6k^2-5k) \lambda^2 }.
\end{align}
Actually, in view of \eqref{eq:gamma:approx} and \eqref{eq:tilr:approx},
we have $\nuE = 2^{-1} \lambda^2 \sum_{k=1}^m r_k k^2 + O(\E)$.

Exact one-point distribution of the \ac{KPZ} equation 
starting from the narrow wedge initial condition 
has been derived and proven by Amir--Corwin--Quastel \cite{amir11},
based on previous work of Tracy and Widom
\cite{tracy09} (see also \cite{tracy08a,tracy08,tracy11}),
and independently obtained by \cite{sasamoto10}. 
Specifically, put $\scrF(T,X) = \log Z_T(X) - \log P_T(X)$.
The law of $\scrF(T,X)$ is given as in \cite[Theorem 1.1]{amir11}.
By using Theorem \ref{thm:main} in a way similar to \cite{amir11}
and using the uniqueness of \eqref{eq:SHE:mild},
we thus obtain in our context that

\begin{theorem}\label{thm:step:stats}
Let $\calZE_T(X)$ be
defined by \eqref{eq:calZE:defn} and \eqref{eq:Z:defn} 
for a weakly asymmetric exclusion process
satisfying \eqref{eq:gamma:approx} and $m\leq 3$.
Consider, in terms of the occupation variable, the step initial condition
\begin{equation}\label{eq:step:ic}
	\eta_0(y) = \bbbone_{(0,\infty)}(y) - \bbbone_{(-\infty,0)}(y).
\end{equation}
For each fixed $(T,X)\in(0,\infty)\times\bbR$, 
\begin{align}\label{eq:F:engy:defn}
	 \FE_T(X) :=
	-\lambda h_{\beta\E^{-2} T} (\beta'\E^{-1}X)
	+ \nuE \E^{-1}\beta T + \log \frac{\lambda\beta'}{2\E^\frac12}
	- \log P_T(X)
\end{align}
converges in distribution to $\scrF(T,X)$ as $\E\to 0$.
\end{theorem}

Note that while there is \emph{no known exact formula} 
for correlation functions or moments for non-simple exclusion processes,
Theorem \ref{thm:main} implies 
exact limiting statistics for the step initial condition
by providing convergence toward the \ac{SHE}.
In the context of \ac{KPZ} universality,
Theorem \ref{thm:step:stats} is 
the first exact limiting statistics result derived 
out of the limiting stochastic partial differential equation,
and \emph{not by explicitly solving the model} in question.

In Section \ref{sec:discr:SHE}, 
we examine the dynamics equation of $\ZE_t$
and obtain an approximated \ac{dSHE},
under the assumptions \eqref{eq:gamma:approx} and $m\le 3$. 
In Sections \ref{sec:conv:SHE}
we establish the following two propositions,
from which Theorem \ref{thm:main} immediately follows.

\begin{proposition}\label{prop:main:tight}
Let $\calZE_T(X)$ be
defined by \eqref{eq:calZE:defn} and \eqref{eq:Z:defn} 
for a weakly asymmetric exclusion process
satisfying \eqref{eq:gamma:approx} and $m\leq 3$.
Assume that the initial condition $\calZE_0\dotBK$ satisfies 
\eqref{eq:initial:cond:bd} and \eqref{eq:initial:cond:Holder},
then the laws of $\curBK{\calZE}_\E$ on $D([0,\infty),C(\bbR))$ is tight.
Moreover, limit points of the law of  $\curBK{\calZE}_\E$
concentrate on $C([0,\infty),C(\bbR))$.
\end{proposition}

\begin{proposition}\label{prop:main:SHE}
Let $\calZE_T(X)$ be
defined by \eqref{eq:calZE:defn} and \eqref{eq:Z:defn} 
for a weakly asymmetric exclusion process
satisfying \eqref{eq:gamma:approx} and $m\leq 3$.
Assume that the initial condition $\calZE_0$ satisfies \eqref{eq:initial:cond:conv}
for some $C(\bbR)$-valued process $Z_0\dotBK$,
then the law of any limit point $\calZ$ of $\curBK{\calZE}$ 
is a solution to \eqref{eq:SHE:mild} starting from $Z_0$.
\end{proposition}

\noindent
Proposition \ref{prop:main:tight} is proven by
applying the H\"older continuities of the microscopic heat kernel to the \ac{dSHE}.
Proposition \ref{prop:main:SHE} is
established by converting the \ac{SHE} \eqref{eq:SHE:mild} to a martingale problem,
and proving that any limit point $\calZ$ of $\curBK{\calZE}$
satisfies the martingale problem.
The key to the proof is Lemma \ref{lem:pre:replace},
which assures certain fluctuation fields weakly converge to zero at the hydrodynamical scale.
Lemma \ref{lem:pre:replace} is proven by standard hydrodynamic-limit type analysis
using the relevant one-block and two-blocks estimates, which we do in Section \ref{sec:replace:lem}.
Finally, in order to apply Theorem \ref{thm:main} to prove Theorem \ref{thm:step:stats},
we need to establish Lemma \ref{lem:exact:stats:ic},
which ensures the H\"older continuity of $\calZE_\cdot\dotBK$
after a short time starting from the step initial condition \eqref{eq:step:ic}.
This is done in Section \ref{sec:exact:stats}.

\subsection*{Acknowledgment.} 
We thank Ivan Corwin, Fraydoun Rezakhanlou, and Stefano Olla
for useful discussions, 
and we are grateful to Tadahisa Funaki 
for the suggestion for eliminating nonlinearity by seeking gradient terms. 

\section{Discrete stochastic heat equation.}
\label{sec:discr:SHE}

Throughout this paper, we assume 
\begin{equation}\label{eq:time:assump:tilT}
	t,s \in [0,\E^{-2}\tilT],
\end{equation}
for arbitrary but fixed $\tilT>0$,
and $C$ denotes a finite positive constant that may change from line to line.
We use $x,x',x_1$ to denote points of $\bbZ$,
where the height function is defined,
and use $y,y',y_1$ to denote points of $\halfZ$,
where the particles are.
For positive integers $i,j,k\in\bbN$, 
we adopt the notation $\bari:=i-\frac12$, $\barj:=j-\frac12$, and $\bark:=k-\frac12$.

We first derive the dynamical equation of $\ZE_t$.
Recall from \eqref{eq:Possion} that $Q^k_\cdot(y)$ are independent
Poisson processes with rate $\qE_k$.
Let $\scrF_t:=\sigma\curBK{Q^k_s(y):s\in[0,t],y\in\halfZ,|k|\le m}$ and
let $a_{y_1\to y_2}$ be the indicator of allowed hops from $y_1$ to $y_2$,
that is
\begin{equation}\label{eq:a:indicator:defn}
	a_{y_1\to y_2}(\eta)
:=
	\frac{1+\eta(y_1)}{2} \frac{1-\eta(y_2)}{2}
	\in \curBK{0,1}.
\end{equation}
From the dynamics of $h$ described in Section \ref{sec:intro}, 
we know that the height at $x$ decreases by 2 
when a particle hops across $x$ from left to right. 
By the definition \eqref{eq:Z:defn} of $\ZE_t$, during $[t,t+dt]$ 
the total contribution to $d\ZE_t(x)$ of hops to the right is
\begin{equation*}
	\BK{ e^{2\lambdaE} -1 } \ZE_t(x) 
	\sum_{k=1}^m \sum_{x-k<y<x} a_{y\to y+k}(\eta_t) dQ^k_t(y).
\end{equation*}
Similarly, the contribution of hops to the left to $d\ZE_t(x)$ is
\begin{equation*}
	\BK{ e^{-2\lambdaE} -1 } \ZE_t(x)
	\sum_{k=1}^m \sum_{x<y<x+k} a_{y\to y-k}(\eta_t) dQ^{-k}_t(y).
\end{equation*}
Aside from these contributions, 
there is a continuous growth of $\ZE_t$ due to $\exp(\E \nuE t)$. 
Gathering the preceding contributions together,
and separating drifts and martingale in each $dQ^{k}_t$,
we obtain the following infinite ($x\in\bbZ$) system of
stochastic differential equations
\begin{equation}\label{eq:Z:governingEq}
	d\ZE_t
= 
	\BK{ \Omega^\varepsilon + \E\nuE  }\ZE_t dt
	+
	\ZE_t d\ME_t,
\end{equation}
where $\ME_t(x)$ is a martingale in $t$ for each $x\in\bbZ$, given by
\begin{equation}\label{eq:M:defn}
\begin{split}
&
	 d\ME_t(x) =
	\BK{ e^{2\lambdaE} -1 } 
	\sum_{k=1}^m \sum_{x-k<y<x} a_{y\to y+k}(\eta_t)
	\BK{ dQ^k_t(y) - \qE_k dt }
	\\
&
	\quad +
	\BK{ e^{-2\lambdaE} -1 } 
	\sum_{k=1}^m \sum_{x<y<x+k} a_{y\to y-k}(\eta_t)
	\BK{ dQ^{-k}_t(y) - \qE_{-k} dt },
\end{split}
\end{equation}
and
\begin{align}\label{eq:Omega:defn}
\begin{split}
	&
	\Omega^\varepsilon_t(x)
	=
	\BK{e^{2\lambdaE}-1}
	\sum_{k=1}^m \qE_k \sum_{x-k<y<x}  a_{y\to y+k}(\eta_t)	
\\
	&
	\quad
	+ 
	\BK{e^{-2\lambdaE}-1}
	\sum_{k=1}^m \qE_{-k} \sum_{x<y<x+k}  a_{y\to y-k}(\eta_t).
\end{split}
\end{align}

For the rest of this section we focus on the drift term $\OmegaE+\E\nuE$.
When $m=1$ under the choice of \eqref{eq:Gartner:para},
the drift is merely a discrete Laplacian
\begin{align*}
	\OmegaE_t + \E\nuE 
	= 2^{-1} \tilrE_1 \Delta_1 \ZE_t,
\end{align*}
where $\tilrE_1 = r_1(1-\E(\gammaE_1)^2)^\frac12$.
For $m>1$, we aim at linearizing the drift term,
to which end we expand the difference
$
	\BK{ \OmegaE+\E\nuE } - \Lap
$
in $\E$, where $\Lap$ is defined as in \eqref{eq:disc:Laplace},
and choose $\gammaE_k$ and $\tilrE_k$ 
as \eqref{eq:gamma:approx} and \eqref{eq:tilr:approx}
to reduces this difference.

Combining \eqref{eq:a:indicator:defn} and\eqref{eq:Omega:defn},
we group terms into $\eta$-linear terms, $\eta$-quadratic terms, and constants.
The $\eta$-linear terms always appear in a symmetric form:
\begin{align}
	\label{eq:Li:defn}
	\calL^\barj_t(x) := 
	\eta_t(x-\barj) - \eta_t(x+\barj). 
\end{align}
Similarly, fixing $k$ and summing over $y\in(x-k,x)$ or $y\in(x,x+k)$ in \eqref{eq:Omega:defn},
all $\eta$-quadratic terms we get are of the form
\begin{align}
	\label{eq:Qd:defn}
	\calQ^k_t(x) := \sum_{\substack{y_1<x<y_2 \\ y_2-y_1 =k}} 
	\eta_t(y_1)\eta_t(y_2).
\end{align}
Note that $-\calQ^k_t(x)$ counts the number of all possible hops of distant $k$ \emph{across} $x$.
After summing over $k$, we get from \eqref{eq:Omega:defn} that
\begin{equation}\label{eq:Omega:expansion}
	\OmegaE_t + \E\nuE = \E^\frac12\Omegali_t + \E\Omegaqd_t + (\nuE - \nuE')\E,
\end{equation}
where
\begin{align}\label{eq:nu':defn}
	\Omegali_t(x)
	:= 
	\sum_{j,k=1}^m \bbbone_\curBK{j\le k} \rhoE_{k} \calL^\barj_t(x),
\quad
	\Omegaqd_t(x) :=
	\sum_{k=1}^m \sigmaE_k  \calQ^k_t(x),
\quad
	\nuE' := \sum_{k=1}^m k \sigmaE_k,
\end{align}
and using \eqref{eq:hop:rate:defn} we have
\begin{align}
	\label{eq:rho:defn}
	\rhoE_{k} 
	&:= 
	4^{-1} \E^{-\frac12}
	\BK{ (e^{2\lambdaE}-1) \qE_k + (1-e^{-2\lambdaE}) \qE_{-k} }
	=
	4^{-1} r_k \BK{ u_k(\E) - \E \gammaE_k v_k(\E) },
\\
	\label{eq:sigma:defn}
	\sigmaE_{k} 
	&:=
	4^{-1} \E^{-1} \BK{ -(e^{2\lambdaE}-1) \qE_k + (1-e^{-2\lambdaE}) \qE_{-k} }
	=
	4^{-1} r_k \BK{ \gammaE_k u_k(\E) - v_k(\E) },
\end{align}
for $u(\E)$ and $v(\E)$ as in \eqref{eq:uv:defn}.

%
Next we Taylor-expand \eqref{eq:Laplace:exp:expansion}.
To accommodate the time evolution up to $\E^{-2}\tilT$,
we need to match \eqref{eq:Omega:expansion} with
\eqref{eq:Laplace:exp:expansion} slightly better than $O(\E^2)$.
Specifically, we neglect terms of the form $\E^2\calE_t$,
where $\calE_t$ is a linear combination of uniformly vanishing and weakly vanishing terms defined as following,
and we also neglect terms of the form $\E \calG_t$,
where $\calG_t$ is a gradient term defined as following.

\begin{definition}\label{defn:calE:gradient}
We say an $\scrF_t$-adopted process $\calE_t(x)$ is \textbf{weakly vanishing} 
if
\begin{align}\label{eq:unifom:bdd}
	\sup_{\E\in(0,1)}\sup_{t,x}
	\VertBK{\calE_t(x)}_\infty < \infty,
\end{align}
and for any $\phi\in C_c(\bbR)$,
any $T>0$, $n=1,2$,
 \begin{align}\label{eq:weakly:vanish}
 	\E^2\int_0^{\E^{-2}T} \BK{ \E\sum_x \phi(\E x) \calE_s(x) \ZE_s(x)^n } ds
 	\Longrightarrow 0.
\end{align}
We say an $\scrF_t$-adopted process $\calE_t(x)$ is \textbf{uniformly vanishing} if
$
	\lim_{\E\to 0}\sup_{t,x}
	\VertBK{\calE_t(x)}_\infty = 0.
$

A process $\calG_t(x)$ is a \textbf{gradient term} if 
$\calG_t = \sum_{|k|\le m} \nabla_k (\calE^k_t \ZE_t)$,
where each $\calE^k_t$ is a linear combination of uniformly vanishing and weakly vanishing terms.
\end{definition}

\noindent
Indeed, when integrating a gradient term against a smooth test function
as done in \eqref{eq:weakly:vanish},
by summation by parts we can move the discrete gradient to the test function
and gain a factor of $\E$.
Hence, we should think of a gradient term as carrying a factor of $\E$. 

In the sequel, we use $\calE_t$ to denote a \emph{generic} term 
that is a linear combination of uniformly vanishing and weakly vanishing terms,
and use $\calG_t$ to denote a \emph{generic} gradient term. 
Our goal is to show

\begin{proposition}\label{prop:main:drift}
For $m\le 3$, under the choice \eqref{eq:gamma:approx}, \eqref{eq:tilr:approx}, and \eqref{eq:nu:defn}
of parameters,
\begin{equation}\label{eq:discr:SHE}
	d \ZE_t =
	\frac12 \sum_{k=1}^m \tilrE_k \Delta_k \ZE_t dt	
	+
	\ZE_t d\ME_t
	+
	\BK{ \E^2 \calE_t + \E\calG_t } \ZE_t dt.
\end{equation}
\end{proposition}

\noindent
To this end, we start by proving

\begin{proposition}\label{prop:drift}
For any $\tilrE_k\in\bbR$,
\begin{equation*}
	\ZE_t \Lap 
	=
	\BK{
		\E^\frac12 \sum_{j=1}^m (A^\E\tilrE)_j \calL^\barj_t
		+ \E \sum_{j=1}^m (B\tilrE)_j \calQ^j_t
		+ \E \nuE''
		+ \E^\frac32 \tilcalC_t
		+ \E^2 \calE_t
		}
	\ZE_t
	+
	\E \calG_t,
\end{equation*}
where $A^\E$ and $B$ are defined as in \eqref{eq:AE:defn} and \eqref{eq:B:defn},
and 
\begin{align}\label{eq:nu'':defn}
	\nuE'' := \frac12 \sum_{k=1}^m \tilrE_k [k (\lambdaE)^2 + \frac{6k^2-5k}{12} (\lambdaE)^4].
\end{align}
Here $\tilcalC_t$ denotes a generic $\eta$-cubic term of the form
\begin{align}\label{eq:tilC:defn}
	\sum_{0<i<j<k\le m} \alphaE_{i,j,k} 
	\Big(
		\eta_t(x+\bari)\eta_t(x+\barj)\eta_t(x+\bark) 
		- \eta_t(x-\bark)\eta_t(x-\barj)\eta_t(x-\bari) 
	\Big),
\end{align}
where $\alphaE_{i,j,k}$ are deterministic and bounded.
\end{proposition}

\begin{remark}\label{rmk:m<=3}
In Proposition \ref{prop:drift:m<=3}
we will show that for $m\le 3$, 
$\E^\frac32\tilcalC_t$ is of the form $\E^\frac32(\calG_t + \E^\frac12\calE_t)$,
hence negligible.
While this is not true when $m>3$,
we conjecture that even then the overall contribution of 
$\E^\frac32\tilcalC$ is still negligible in the limit as $\E\to 0$.
\end{remark}

The following lemma, whose proof is deferred to Section \ref{sec:replace:lem},
is needed for proving Proposition \ref{prop:drift}.

\begin{lemma}\label{lem:pre:replace}
For any fixed distinct $(y_1,\ldots,y_{n_0})\in\halfZ^j$, $n_0=1,\ldots,4$, the term 
\begin{align}\label{eq:Phi:defn}
	\Phi_t(x) := \prod_{i=1}^{n_0} \eta_t(x+y_i)
\end{align}
is weakly vanishing.
\end{lemma}

\begin{remark}\label{rmk:vanishing}
Since $\Phi_t(x)=\pm 1$,
it does not vanish uniformly.
Yet we expect it to vanish weakly because we at near equilibrium fluctuation.
More precisely,
for any $\alpha\in[0,1]$, 
consider the product measure $\nu_a$ on $\curBK{\pm 1}^\halfZ$ of 
the i.i.d.\ Bernoulli measures $\pi_a(1)=a$, $\pi_a(-1)=1-a$,
which is an invariant measure of exclusion processes.
The initial condition \eqref{eq:initial:cond:bd} corresponds to fluctuations near $\nu_{1/2}$.
Since $\nu_{1/2}(\prod_{i=1}^n\eta(x+y_i))=0$,
we expect $\prod_{i=1}^n\eta_t(x+y_i)$ to be small after being averaged over 
a large spacetime section.
\end{remark}

\begin{proof}[Proof of Proposition \ref{prop:drift}]
Taylor-expand the exponential functions in \eqref{eq:Laplace:exp:expansion}
up to the fourth order to get $\sum_{n=1}^4 \E^\frac{n}{2} D_n + \E^\frac52 R$,
where $D_n$ is a linear combination of terms of the form
$(-\sum_{y\in(x,x+k)} \eta_t(y))^n$ and $(\sum_{y\in(x-k)} \eta_t(y))^n$,
and $R$ is a uniformly bounded remainder.
Generically, 
$D_1$ consists of $\eta$-linear terms, 
$D_2$ consists of $\eta$-quadratic terms,
$D_3$ consists of $\eta$-cubic terms,
$D_4$ consists of $\eta$-quartic terms, respectively,
but since $\eta_t(y)^2=1$ we also get constants in $D_2$,
$\eta$-linear terms in $D_3$, 
and $\eta$-quadratic terms and constants in $D_4$.
By Lemma \ref{lem:pre:replace}, 
the non-constant terms of $D_4$ are weakly vanishing,
and clearly $\E^\frac12 R$ is uniformly vanishing.
Hence the sum of all non-constant terms 
in last two terms $\E^2(D_4+\E^\frac12 R)$ of the Taylor-expansion 
is of the type $\E^2\calE_t$.
(We will repeatedly use this fact in this proof when doing Taylor-expansion 
without explicitly stating it.)
Gathering the constants in $D_2$ and $D_4$,
and combining the $\eta$-linear terms in $D_3$ with $D_1$,
we then obtain  
\begin{align*}
	(\ZE_t)^{-1} \Lap
=
	\E^\frac12 D^\text{lin}
	+ \E D^\text{qd} + \E^\frac32 D^\text{cub}
	+ \nuE''\ZE + \E^2 \calE,
\end{align*}
where 
\begin{align}
	\label{eq:D:lin:defn}
	 D^\text{lin}_t(x)
	 &:=
	\sum_{j,k=1}^m 
	\BK{ \frac{\lambda}{2} + \E\lambda^3 \frac{3k-2}{12} }
	\bbbone_\curBK{j\le k}
	\tilrE_{k} \calL^\barj_t(x), 
\\
	\label{eq:D:qd:defn}
	 D^\text{qd}_t(x) 
	 &:=
	\frac{\lambda^2}{2}
	\sum_{k=1}^m \tilrE_k 
	\sum_{(y_1,y_2)\in\calI^{k}_2(x+) \cup \calI^{k}_2(x-)}
	\eta_t(y_1)\eta_t(y_2),
\\
	\label{eq:D:cub:defn}
	 D^\text{cub}_t(x) 
	 &:=
	\frac{\lambda^3}{2}	
	\sum_{k=1}^m \tilrE_k 
	\BK{
		\sum_{(y_1,y_2,y_3)\in\calI^{k}_3(x+)} 
		\hspace{-10pt} \eta_t(y_1)\eta_t(y_2)\eta_t(y_3)
		- \sum_{(y_1,y_2,y_3)\in\calI^{k}_3(x-)} 
		\hspace{-10pt} \eta_t(y_1)\eta_t(y_2)\eta_t(y_3)
		},
\end{align}
and
\begin{align*}
	\calI^{n}_k(x+) 
	&:= 
	\curBK{ 
		(y_1,\ldots,y_n): 
		y_1<\ldots<y_n \in \halfZ \cap (x,x+k) 
		},
\\
	\calI^{n}_k(x-) 
	&:= 
	\curBK{ 
		(y_1,\ldots,y_n): 
		y_1<\ldots<y_n \in \halfZ \cap (x-k,x) 
		}.
\end{align*}
The terms $D^\text{qd}$ is the sum of the signs of all 
possible hops within $(x,x+k)$ and within $(x-k,x)$,
and $D^\text{cub}$ is the difference of two sums,
consisting of signs corresponding to non-degenerated 
(distinct coordinates) cubic terms within $(x,x+k)$ and within $(x-k,x)$.
The factor $3k-2$ in \eqref{eq:D:lin:defn} 
counts the number of degenerated cubic terms with one coordinate being equal to a given value,
the factor $k$ in \eqref{eq:nu'':defn} counts the total number of degenerated quadratic terms,
and the factor $6k-5$ in \eqref{eq:nu'':defn} counts the number of 
of degenerated quartic terms such that two of its coordinates take the same value 
and the other two coordinates also take the same value 
(which can be the same or different from the value of the previous two coordinates).

The quadratic term $D^\text{qd}$ corresponds to hops \emph{not across} $x$,
whereas $\Omegaqd$ corresponds to hops \emph{across} $x$.
Hence, we match $D^\text{qd}$ with $\Omegaqd$ by
translating the center point $x$ to lie between the two particles.
That is, we rewrite a generic term $\calQ:=\eta_t(y_1)\eta_t(y_2)\ZE_t(x)$, 
$(y_1,y_2)\in\calI_k^2(x+)$, of $D^\text{qd}$ as
\begin{align}\label{eq:prop:drift:2}
	\calQ = \eta_t(y_1) \eta_t(y_2) \ZE_t(x+i)
	+ \eta_t(y_1) \eta_t(y_2) \ZE_t(x) \BK{ 1 - \ZE_t(x+i)/\ZE_t(x) },
\end{align}
and choose $i\in I(y_1,y_2):=\bbZ\cap(y_1-x,y_2-x)$.
Since, the first term of \eqref{eq:prop:drift:2} is a translation of 
$\eta_t(y_1-i) \eta_t(y_2-i) \ZE_t(x)$,
we write it as the sum of a gradient term and
a generic term $\eta_t(y_1-i) \eta_t(y_2-i) \ZE_t(x)$ of $\Omegaqd$.
By \eqref{eq:Z:defn}, $\ZE_t(x+i)/\ZE_t(x)$ is an exponential function of the height difference.
By Taylor-expanding the exponential function to the first order
we obtain $\eta$-linear terms and some remainders.
The $\eta$-linear terms combined with $\eta_t(y_1)\eta_t(y_2)$ yield $\eta$-cubic terms,
except when the $\eta$-linear term coincides with $\eta_t(y_1)$,
where we have $\eta_t(y_2)$.
Thus we obtain
\begin{align}\label{eq:prop:drift:3}
	\calQ =
	\calG_t(x) +
	\BK{
		\eta_t(y_1-i)\eta_t(y_2-i)
		+ \lambda \E^\frac12 \eta_t(y_2)
		+ \E^\frac12 \calC^{(0)}_t
		+ \E \calE_t(x)
		}
	\ZE_t(x),
\end{align}
where $\calC^{(0)}_t$ is a sum of $\eta$-cubic terms, and $i\in I(y_1,y_2)$.
Notice that $\calQ^j_t(x)$, as defined in \eqref{eq:Qd:defn}, 
is the sum of $\eta_t(y_1-i)\eta_t(y_2-i)$ over $i\in I(y_1,y_2)$,
where $j=y_2-y_1$ is fixed.
Hence, by summing \eqref{eq:prop:drift:3} over $i\in I(y_1,y_2)$,
we obtain
\begin{align}\label{eq:prop:drift:4}
	j \eta_t(y_1)\eta_t(y_2)\ZE_t(x) 
= 
	\calG_t(x) +
	\BK{
		\calQ^j_t(x)
		+ j \lambda \E^\frac12 \eta_t(y_2)
		+ \E^\frac12 \calC_t
		+ \E \calE_t(x)
		}
	\ZE_t(x),
\end{align}
where $\calC_t$ is a sum of $\eta$-cubic terms.
Applying the same reasoning to the mirror image $(y_2',y_1')$ 
of $(y_1,y_2)$ with respect to $x$,
namely $(y_2',y_1')=(2x-y_2,2x-y_1)$, we get 
\begin{align}\label{eq:prop:drift:5}
	j \eta_t(y_2')\eta_t(y_1')\ZE_t(x) 
= 
	\calG_t(x) +
	\BK{
		\calQ^j_t(x)
		- j \lambda \E^\frac12 \eta_t(y_2')
		- \E^\frac12 \calC'_t
		+ \E \calE_t(x)
		}
	\ZE_t(x),
\end{align}
where $\calC'_t$ is a sum of $\eta$-cubic terms,
which, by symmetry, is the mirror image of $\calC_t$ with respect to $x$
(that is, $\calC_t-\calC'_t$ is of the form $\tilcalC_t$ as in \eqref{eq:tilC:defn}).
Combining \eqref{eq:prop:drift:4} and \eqref{eq:prop:drift:5} we obtain
\begin{align}\label{eq:qd:trans}
\begin{split}
	&
	2^{-1}
	\big( \eta_t(y_1)\eta_t(y_2) + \eta_t(y_2') \eta_t(y_1') \big)
\\
	&
	\quad
	=
	\calG_t(x) +
	\BK{
		j^{-1} \calQ^j_t(x)
		- 2^{-1} \lambda \E^\frac12 \calL^{y_2-x}_t(x)
		- \E^\frac12 \tilcalC_t(x)
		+ \E \calE_t(x)
		}
	\ZE_t(x).
\end{split}
\end{align}
Since the set $\calI_2^k(x-)$ is the mirror image of $\calI_2^k(x+)$
with respect to $x$,
we can apply \eqref{eq:qd:trans} to \eqref{eq:D:qd:defn}.
After rearranging the sum over $y_1$ and $y_2$, we obtain
\begin{align}\label{eq:prop:drift:6}
	D^\text{qd}_t
=
	\lambda^2 \sum_{j,k=1}^m
	\calQ_t^l(x) \tilrE_k (k-j)_+ j^{-1}
	-
	2^{-1} \E^\frac12 \lambda
	\sum_{j,k=1}^m \tilrE_k (j-1) \bbbone_\curBK{j\le k}
	\calL^\barj_t
	+
	\calG_t
	+
	\E^\frac12 \tilcalC_t.
\end{align}
Combining the second term of \eqref{eq:prop:drift:6} with $D^\text{lin}$, we obtain
\begin{align}\label{eq:D:lin:qd:AB}
	\E^\frac12 D^\text{lin}_t + \E D^\text{qd}_t
=
	\E^\frac12 \sum_{j=1}^m (A^\E\tilrE)_j \calL^\barj_t
	+
	\E \sum_{j=1}^m (B\tilrE)_j \calQ^j_t.
\end{align}

To conclude the proof, it thus suffices to show 
that $D^\text{cub}$ is of the form $\E^\frac12 \calE_t+\calG_t$.
To this end, we employee a translation similar to \eqref{eq:prop:drift:2}.
For each cubic term $\eta_t(y_1)\eta_t(y_2)\eta_t(y_3) \ZE_t(x)$,
$(y_1,y_2,y_3)\in\calI^{3}_k(x+)$,
translate the center from $x$ to $x+k$ to get
\begin{align}\label{eq:prop:drift:1}
\begin{split}
	\eta_t(y_1)\eta_t(y_2)\eta_t(y_3) \ZE_t(x)
	&=
	\calG_t(x)
	+
	\eta_t(y_1-k)\eta_t(y_2-k)\eta_t(y_3-k) \ZE_t(x)
\\
	&
	\quad\quad
	+ \eta_t(y_1)\eta_t(y_2)\eta_t(y_3) 
	\ZE_t(x) \BK{ 1 - \ZE_t(x)/\ZE_t(x+k) }.
\end{split}
\end{align}
For the last term in \eqref{eq:prop:drift:1},
using \eqref{eq:Z:defn} and Taylor-expansion to the first order,
we turn it into the form $\E^\frac12\calE$.
Since $\calI^{3}_k(x+)\rightarrow \calI^{3}_k(x-)$: 
$(y_1,y_2,y_3)\mapsto (y_1-k,y_2-k,y_3-k)$ is a bijection,
the sum of $\eta_t(y_1-k)\eta_t(y_2-k)\eta_t(y_3-k)$ over $I^3_k(x+)$
in \eqref{eq:prop:drift:1} matches the sum over $\calI^k_3(x-)$ in \eqref{eq:D:cub:defn}.
Consequently, $D^\text{cub} = \E^\frac12 \calE_t + \calG_t$, as claimed.
\end{proof}

\begin{proposition}\label{prop:drift:m<=3}
For $m\le 3$,
$\tilcalC_t=\E^\frac12 \calE_t+\calG_t$.
\end{proposition}

\begin{proof}
Clearly, $\tilcalC_t=0$ when $m=1,2$,
whereas when $m=3$ \eqref{eq:tilC:defn} consists of the single term
corresponding to $i=1$, $j=2$, $k=3$.
Since $(x+\frac12,x+\frac32,x+\frac52)$ is a translation by $3$
of $(x-\frac52,x-\frac32,x-\frac12)$,
the argument in the last paragraph 
of the proof of Proposition \ref{prop:drift} also applies to $\tilcalC$,
concluding the proof.
\end{proof}

We now combine Proposition \ref{prop:drift} and \ref{prop:drift:m<=3} 
to prove Proposition \ref{prop:main:drift}. 

\begin{proof}[Proof of Proposition \ref{prop:main:drift}]
First, comparing \eqref{eq:nu:defn} with \eqref{eq:nu':defn} and \eqref{eq:nu'':defn},
we find that the constants always match, that is $\nuE=\nuE'+\nuE''$.
Next, by \eqref{eq:Omega:expansion} and \eqref{eq:D:lin:qd:AB},
the equation \eqref{eq:matching:eq} implies 
\begin{align}\label{eq:match:li:qd}
	\E^\frac12 D^\text{lin}+ \E D^\text{qd} 
	=
	\E^\frac12 \Omegali + \E \Omegaqd + O(\E^2)R, 
\end{align}
where $R$ is a linear combination of $\eta$-linear and $\eta$-quadratic terms,
which by Lemma \ref{lem:pre:replace} is weakly vanishing.
Hence the remainder $O(\E^2)R$ is of the desired form $\E^2 \calE_t$.
Proposition \ref{prop:main:drift} now follows from Proposition \ref{prop:drift} and \ref{prop:drift:m<=3}.
\end{proof}

\section{Convergence to the \ac{SHE}}
\label{sec:conv:SHE}

Let $\bfp$ be the kernel of the following semi-discrete heat equation
\begin{equation}\label{eq:discr:heat:kernel}
\begin{split}
	&
	\frac{d~}{dt} \bfp_t(x) = \frac12 \sum_{k=1}^m \tilrE_k \Delta_k \bfp_t(x),
\\
	&
	\bfp_0(x) = \bbbone_\curBK{0}(x). 
\end{split}
\end{equation}
Let $*$ denote the convolution of two functions on $\bbZ$, that is
$(f*g)(x):=\sum_{x'} f(x-x')g(x')$.
We rewrite the \ac{dSHE} \eqref{eq:discr:SHE}
as the following integrated form:
\begin{align}\label{eq:discr:SHE:int}
	\ZE_t = \bfp_t * \ZE_0 + \int_0^t \bfp_{t-s} * (\ZE_s d\ME_s)
	+ \int_0^t \E^2 \bfp_{t-s} * (\calE_s\ZE_s) ds 
	+ \sum_{|k|\le m} \int_0^t \E\nabla_k \bfp_{t-s} * (\calE^k_s\ZE_s) ds, 
\end{align}
where we applied summation by parts to the last term.
Here, as in Definition \ref{defn:calE:gradient},
each $\calE^k$ is a linear combination of uniformly vanishing and weakly vanishing terms.

\subsection{Tightness}

In this section we prove Proposition \ref{prop:main:tight}.
The key to the proof is the H\"older estimates of $\ZE_t$ 
given in Proposition \ref{prop:Holder} and Corollary \ref{cor:Holder},
whose proofs require the following

\begin{lemma}\label{lem:QVcomparision}
Given any deterministic function $f_s(x,x')$: $[0,\infty)\times\bbZ^2\to\bbR$, let 
\begin{align*}
	\tilf_s(x,x')
	:=
	\sup\curBK{ 
		\absBK{ f_{s'}(x,x') f_{s'}(x,x'+j) } :
		\floorBK{s}  \le s' < \floorBK{s}+1, |j|<m
		}.
\end{align*}
For any $n\in\bbN$ we have
\begin{align*}
	\Big\Vert \int_t^{t'} \sum_{x'} f_s(x,x') \ZE_s(x') d\ME_s(x') \Big\Vert_{2n}^2
&\leq
	C \E \int^{t'}_t \sum_{x'} \tilf_s(x,x')
	\, \VertBK{\ZE_s(x')^2}_n ds.
\end{align*}
\end{lemma}

\begin{proof}
Fix $t$ and let 
$R_{t'}(x):=\int_t^{t'} \sum_y f_s(x,x') \ZE_s(x') d\ME_s(x')$.
By the Burkholder--Davis--Gundy inequality, 
\begin{equation}\label{eq:lem:QVcomparision:1}
	\VertBK{ R_{t'}(x)^2 }_n
\leq
	C \VertBK{ [R_\cdot(x), R_\cdot(x)]_{t'} }_n,
\end{equation}
where $[\,\cdot\,,\,\cdot\,]$ denotes the quadratic variation.
Let $\fkT_{(s_1,s_2]}(x)$ be the (random) set of all
$s\in(s_1,s_2]$ at which a particle hops across the site $x$.
Since the Poission processes 
as defined in \eqref{eq:Possion} are mutually independent,
using \eqref{eq:M:defn} we have
\begin{align*}
	[R_\cdot(x), R_\cdot(x)]_{t'}
=
	\sum_{x'} \sum_{|l|<m} 
	\sum_{s\in\fkT^l(x')} f_{s^{-}}(x,x') f_{s^{-}}(x,x'+l) 
	(e^{2\sigma(x',l)\lambdaE}-1)^2 \ZE_{s^{-}}(x') \ZE_{s^{-}}(x'+l),
\end{align*}
where $\fkT^l(x):=\fkT_{(t,t']}(x)\cap \fkT_{(t,t']}(x+l)$ is the set of $s\in(t',t]$
at which a particle hops across \emph{both} the site $x$ and $x+l$, 
and $\sigma(x,l)=1$ when the particle hops to the right, 
$\sigma(x,l)=-1$ when the particle hops to the left.
(Note that the sum of $l$ goes over $|l|<m$ since $\fkT^l(x)=\emptyset$ when $|l|\geq m$.)
Next we partition $[0,\infty)$ into subintervals $\calT_i:=[i,i+1)$.
Using $|e^{\pm\lambdaE}-1| \le C \E^{1/2}$,
and replacing $f_s$ and $\ZE_s$ by their supremum over $\calT_i$, we have
\begin{align}\label{eq:lem:QVcomparision:2}
    [R_\cdot(x), R_\cdot(x)]_{t'}
\leq
     C \E \sum_{i=\floorBK{t}}^{\floorBK{t'}} \sum_{x'}
	\tilf_i(x,x') 
	\sup_{s\in\calT_i} \max_{|l|<m} \ZE_s(x'+l)^2.
\end{align}

Since $|h_t(x+l)-h_t(x)| \le l$, 
from the definition \eqref{eq:Z:defn} of $\ZE_t$, 
\begin{equation}\label{eq:lem:QVcomparision:0}
	\max_{|l|<m} \ZE_s(x'+l)^2 \le C \ZE_s(x')^2.
\end{equation}
Next, since each hop across $x$ 
increases or decreases $\ZE_t(x)$ by a factor of $e^{2\lambdaE}$,
\begin{equation}\label{eq:lem:QVcomparision:4}
	\sup_{s\in\calT_i} \ZE_s(x')^2
	\leq
	e^{ 4\lambdaE N_i(x')} \ZE_i(x')^2,
\quad
	\ZE_i(x')^2
	\leq
	e^{ 4\lambdaE N_i(x')} 
	\inf_{s\in\calT_i} \ZE_s(x')^2
	,
\end{equation}
where $N_i(x')$ is the number of hops across $x'$ during the time interval $\calT_i$, that is
\begin{align}\label{eq:Qiy:defn}
	 N_i(x') : = \# \fkT_{\calT_i}(x'),
\end{align}
which is stochastically bounded by a 
Poisson random variable with rate $\sum_{k=1}^m kr_k$.
Take the $L^n$-norm of \eqref{eq:lem:QVcomparision:4}
using the independence of $N_i(x')$ and $Z^\varepsilon_{\varepsilon^{2}i}(x')$
and the bound $\bbE(e^{n\lambdaE N_i(x')}) \le C(n)$.
We then deduce
\begin{equation}\label{eq:lem:QVcomparision:3}
	\Big\Vert \sup_{s\in\calT_i} \ZE_s(x')^2 \Big\Vert_n
\leq
	C \ \Big\Vert \inf_{s\in\calT_i} \ZE_s(x')^2 \Big\Vert_n
\leq
	C  \left\Vert |\calT_i|^{-1} \int_{\calT_i} \ZE_s(x')^2  ds \right\Vert_n
\leq
	C \int_{\calT_i} \VertBK{ \ZE_s(x')^2 }_n ds.
\end{equation}
Combining \eqref{eq:lem:QVcomparision:1}, \eqref{eq:lem:QVcomparision:2}, 
\eqref{eq:lem:QVcomparision:0}, and \eqref{eq:lem:QVcomparision:3},
we conclude the lemma.
\end{proof}

For process $f\dotBK$ over $\bbZ$, 
$a\ge 0$, $l\in\bbN$, define the following norm
\begin{equation}\label{eq:norm:||}
	|f|_{a,l} := \sup_x \VertBK{f(x)}_l e^{-a\E|x|}.
\end{equation}
Note that $|f^2|_{2a,l}=|f|_{a,2l}^2$ and for any $x$, $x'$ we have
\begin{align}\label{eq:prop:Holder:norm}
	\VertBK{f(x')}_l \le e^{a\E|x|} e^{a\E|x-x'|} |f|_{a,l}.  
\end{align}

\begin{proposition}\label{prop:Holder}
For any $u\in (0,1)$, $a\geq 0$, $j\in\bbN$, $t,t'\in[0,\tilT\E^{-2}]$,
\begin{align}
	\label{eq:Z:moment}
	|\ZE_t|_{a,2j} &\le C(a) |\ZE_0|_{a,2j},
\\
	\label{eq:Z:Holder:esti:space}
	\VertBK{\ZE_t(x)-\ZE_t(x')}_{2j} &\leq
	C(a,u) (\E|x-x'|)^\frac{u}{2} e^{a(\E|x|+|x'|)} (|\ZE_0|_{a,2j}+N_{a,2j,u/2}) ,
\\
	\label{eq:Z:Holder:esti:time}
	\VertBK{ \ZE_t(x) - \ZE_{t'}(x) }_{2j} &\leq
	 C(a,u) (1\vee|t'-t|^\frac{u}{4}) \E^\frac{u}{2} e^{a\E|x|} (|\ZE_0|_{a,2j}+N_{a,2j,u/2}),
\end{align}
where
\begin{align*}
	N_{a,j,v} := \sup_{x\neq x'} |\E(x-x')|^{-v} 
	\VertBK{\ZE_0(x)-\ZE_0(x')}_j e^{-a\E(|x|+|x'|)}.
\end{align*}
\end{proposition}

\begin{proof}
Let $u\in (0,1)$ $a>0$ be given and fixed, 
so that we do not specify the dependence of $C$ on $u$ and $a$.
Let $I_1$, $I_2$, $I_3$, $I_4$ denote the first, second, third, fourth terms 
on the RHS of \eqref{eq:discr:SHE:int}, respectively.

We first prove \eqref{eq:Z:moment}.
Using the readily verified inequality $(\sum_{i=1}^4 b_i)^2 \leq 4\sum_{i=1}^4 b_i^2$,
we get
\begin{align}\label{eq:prop:Holder:4inequ}
	|(\ZE_t)^2|_{2a,j}
\leq
	4 \BK{ 
		|I_1^2|_{2a,j} + |I_2^2|_{2a,j}
		+|I_3^2|_{2a,j} + |I_4^2|_{2a,j} 
		}.
\end{align}
For $I_1$ we have
$
	\Vert I_1^2\Vert_j = \Vert I_1\Vert_{2j}^2
	\leq \BK{ \bfp_t * \Vert\ZE_0\Vert_{2j} }^2,
$
which by using \eqref{eq:prop:Holder:norm} for $l=2j$
and \eqref{eq:p:esti:sum} yields $|I_1^2|_{2a,j} \le C |(\ZE_0)^2|_{2a,j}$.
Next, apply Lemma \ref{lem:QVcomparision} to $\Vert I^2_2\Vert_j$
for $f_s(x,x')=\bfp_{t-s}(x-x')$,
and then use \eqref{eq:p:esti:sup} and \eqref{eq:p:esti:Linfty}.
We then obtain
\begin{align}\label{eq:prop:Holder:I2:Vert}
	\Vert I^2_2 \Vert_j
\leq
	C \E \int_0^t (t-s)^{-\frac12}
	\BK{ \bfp_\ceilBK{t-s}*\VertBK{(\ZE_s)^2}_j } ds.
\end{align}
By applying \eqref{eq:prop:Holder:norm} for $l=2j$ and \eqref{eq:p:esti:sum},
we bound the convolution in \eqref{eq:prop:Holder:I2:Vert} by  
$C e^{2a\E|x|} |(\ZE_s)^2|_{2a,j}$, yielding
\begin{align}\label{eq:prop:Holder:I2:abs}
	|I^2_2|_{2a,j} 
\leq 
	C \E \int_0^t (t-s)^{-\frac12} |(Z_s)^2|_{2a,j} ds.
\end{align}
As for $I_3$ and $I_4$,
by Definition \ref{defn:calE:gradient} we have 
$\sup_{\E,x,t}\VertBK{\calE_t(x)}_\infty,\sup_{\E,x,t}\VertBK{\calE^k_t(x)}_\infty \le C$,
yielding
\begin{align*}
	\VertBK{I_3^2}_j = \VertBK{I_3}_{2j}^2
	\leq
	C \BK{\int_0^t \E^2 \bfp_{t-s} * \VertBK{\ZE}_{2j} ds}^2,
\quad
	\VertBK{I_4^2}_j 
	\leq
	C \sum_{k\le m}
	\BK{\int_0^t \E |\nabla_k\bfp_{t-s}| * \VertBK{\ZE}_{2j} ds}^2.
\end{align*}
Further apply the Cauchy--Schwartz inequality to get 
\begin{subequations}\label{eq:prop:Holder:I34:CS}
\begin{align}
	&
	\label{eq:prop:Holder:I3:CS}
	\VertBK{I_3^2}_j \leq
	 C
	\BK{\int_0^t \E^2 \bfp_{t-s} * 1 ds }
	\BK{ \int_0^t \E^2 \bfp_{t-s} * \VertBK{(\ZE_s)^2}_j ds },
\\
	&
	\label{eq:prop:Holder:I4:CS}
	\VertBK{I_4^2}_j \leq
	 C
	\sum_{|k|\le m}
	\BK{\int_0^t \E |\nabla_k\bfp_{t-s}|*1 ds }
	\BK{ \int_0^t \E |\nabla_k\bfp_{t-s}| * \VertBK{(\ZE_s)^2}_j ds },
\end{align}
\end{subequations}
where $1$ denotes the constant function $f(x)\equiv 1$.
The first integral in \eqref{eq:prop:Holder:I3:CS} is clearly bounded by $\tilT$,
and by \eqref{eq:del:p:esti:sum} for $v=1$
the first integral in \eqref{eq:prop:Holder:I4:CS} is bounded by $ C\tilT^\frac12$.
Hence 
\begin{align}
	&
	\label{eq:prop:Holder:I3:Vert}
	\VertBK{I_3^2}_j \leq
	 C
	\int_0^t \E^2 \bfp_{t-s} * \VertBK{(\ZE_s)^2}_j ds,
\\
	&
	\label{eq:prop:Holder:I4:Vert}
	\VertBK{I_4^2}_j \leq
	 C
	\sum_{|k|\le m} \int_0^t \E |\nabla_k\bfp_{t-s}| * \VertBK{(\ZE_s)^2}_j ds.
\end{align}
Applying \eqref{eq:p:esti:sum} and \eqref{eq:del:p:esti:sum} for $v=1$,
we obtain
\begin{align*}
	|I_3^2|_{2a,j} \le C \int_0^t \E^2 |(\ZE_s)^2|_{2a,j} ds,
\quad
	|I_4^2|_{2a,j} 
	\le C \int_0^t \E (t-s)^{-\frac12} |(\ZE_s)^2|_{2a,j} ds.
\end{align*}
Combining the preceding estimates of 
$|I^2_i|_{2a,j}$, $i=1,\ldots,4$,
we arrive at the following inequality
\begin{align}\label{eq:prop:Holder:iterate}
	|(\ZE_t)^2|_{2a,j} \le C |(\ZE_0)^2|_{2a,j} + C \int_0^t f_\E(t-s) |(\ZE_s)^2|_{2a,j} ds,
\end{align}
where $f_\E(s):= \E s^{-1/2} + \E^2$.
Iterate \eqref{eq:prop:Holder:iterate} to get
\begin{align*}
	|(\ZE_s)^2|_{2a,j} \le  
	|(\ZE_0)^2|_{2a,j} \BK{ C + \sum_{l=1}^\infty \frac{C^l}{l!} \BK{ \int_0^t f_\E(s) ds }^l }.
\end{align*}
Since $t\le \E^{-2}\tilT$, we have
$\int_0^t f_\E(s) ds \le C$, concluding \eqref{eq:Z:moment}.

Next we show \eqref{eq:Z:Holder:esti:space}.
Put $n=x'-x$ so that $\ZE_t(x')-\ZE_t(x)=\nabla_n\ZE_t(x)$.
By \eqref{eq:discr:SHE:int} we have
\begin{align*}
	\Vert \nabla_n\ZE_t \Vert_{2j}
\leq
		\Vert\nabla_nI_1\Vert_{2j} + \Vert\nabla_n I_2\Vert_{2j}
		+\Vert\nabla_n I_3\Vert_{2j} + \Vert\nabla_n I_4\Vert_{2j}.
\end{align*}
For $\nabla_nI_1$,
using summation by parts to move the discrete gradient to $\ZE_0$,
we have
$
	\Vert\nabla_n I_1\Vert_{2j} 
	\leq \bfp_t * \Vert\nabla_n \ZE_0\Vert_{2j}.
$
This together with \eqref{eq:p:esti:sum} implies
\begin{equation}\label{eq:prop:Holder:I1:del}
	\Vert \nabla_n I_1(x) \Vert_{2j} \le 
	|n\E|^{u/2} e^{a\E(|x|+|x'|)} C N_{a,2j,u/2}.
\end{equation}
Next, similar to \eqref{eq:prop:Holder:I2:Vert},
applying Lemma \ref{lem:QVcomparision} to $\Vert(\nabla_n I_2)^2\Vert_j$
for $f_s(x,y)=\nabla_n\bfp_{t-s}(x-y)$,
and using $(\nabla_n\bfp_{t-s})^2 \le |\nabla_n\bfp_{t-s}| (\bfp_{t-s}+\tau_n\bfp_{t-s})$,
and then combining \eqref{eq:p:esti:sup} and \eqref{eq:del:p:esti:Linfty}, we get
\begin{align}\label{eq:prop:Holder:I2:del:Vert}
	\Vert\nabla_n I_2(x)\Vert_{2j}^2
\leq 
	C \E \int_0^t |n|^u (t-s)^{-\frac{1+u}{2}}
	\BK{ \bfp_\ceilBK{t-s} + \tau_n\bfp_\ceilBK{t-s}} * \VertBK{(\ZE_s)^2}_{j} ds.
\end{align}
Using \eqref{eq:p:esti:sum}, \eqref{eq:prop:Holder:norm}, and \eqref{eq:Z:moment}, 
and calculating the time integral,
we then turn \eqref{eq:prop:Holder:I2:del:Vert} into
\begin{align*}
	\VertBK{\nabla_n I_2(x)}_{a,2j}^2
\leq 
	C |\E n|^u e^{2a(\E|x|+\E|x'|)} |(\ZE_0)^2|_{2a,j}.
\end{align*}
For $\nabla_nI_3$ and $\nabla_nI_4$, similar to \eqref{eq:prop:Holder:I34:CS} we have
\begin{align}
	&
	\label{eq:prop:Holder:I3:del:CS}
	\VertBK{\nabla_n I_3(x)}^2_{2j} \leq
	C
	\int_0^t \E^2 |\nabla_n\bfp_{t-s}| * 1 ds
	\
	\int_0^t \E^2 (\bfp_{t-s} + \tau_n\bfp_{t-s}) *
	\VertBK{(\ZE_s)}^2_{2j} ds,
\\
	&
	\label{eq:prop:Holder:I4:del:CS}
	\VertBK{\nabla_n I_4(x)}^2_{2j} \leq
	C
	\sum_{|k|\le m}
	\int_0^t \E |\nabla_n\nabla_k\bfp_{t-s}| * 1 ds
	\
	\int_0^t \E (|\nabla_k\bfp_{t-s}| + |\tau_n\nabla_k\bfp_{t-s}| ) *
	\VertBK{(\ZE_s)}^2_{2j} ds.
\end{align}
By using \eqref{eq:del:p:esti:sum} for $v=u$ and \eqref{eq:del:del:p:esti:sum} for $v=u$ 
to estimates the first integral in \eqref{eq:prop:Holder:I4:del:CS} and \eqref{eq:prop:Holder:I3:del:CS},
we further obtain
\begin{align} 
	&
	\label{eq:prop:Holder:I3:del}
	\VertBK{\nabla_n I_3(x)}^2_{2j} \leq
	C |n\E|^u
	\int_0^t \E^2 (\bfp_{t-s} + \tau_n\bfp_{t-s}) *
	\VertBK{(\ZE_s)}^2_{2j} ds,
\\
	&
	\label{eq:prop:Holder:I4:del}
	\VertBK{\nabla_n I_4(x)}^2_{2j} \leq
	C |n\E|^u
	\sum_{|k|\le m} 
	\int_0^t \E (|\nabla_k\bfp_{t-s}| + |\tau_n\nabla_k\bfp_{t-s}| ) *
	\VertBK{(\ZE_s)}^2_{2j} ds.
\end{align}
Using \eqref{eq:p:esti:sum}, \eqref{eq:del:p:esti:sum} for $v=1$, \eqref{eq:prop:Holder:norm}, and \eqref{eq:Z:moment},
we then bound $\Vert \nabla_n I_3\Vert_{2j}^2$ and $\Vert \nabla_n I_4\Vert_{2j}^2$
by $ C[|n\E|^{u/2} e^{a\E(|x|+|x'|)} |(\ZE_0)|_{a,2j}]^2$.
Combining the preceding estimates for $\Vert \nabla_n I_i\Vert_{2j}^2$, $i=1,\ldots,4$,
we conclude \eqref{eq:Z:Holder:esti:space}.

Next we prove \eqref{eq:Z:Holder:esti:time}.
Without lost of generality, assume $t'>t$.
It suffices to show for $i=1,\ldots,4$,
\begin{align}\label{eq:prop:Holder:time}
	\Vert (I_i)_{t'}(x) - (I_i)_t(x) \Vert_{2j}
\leq
	 (1\vee(t'-t)^{u/4})\E^{u/2} e^{2a\E|x|} C \BK{ |\ZE_0|_{a,2j} + N_{a,2j,u/2} }.
\end{align}
For $i=1$, using the semi-group properties $\bfp_{t'}=\bfp_{t'-t}*\bfp_t$
and $\sum_{x_1} \bfp_{t'-t}(x_1)=1$ we have 
\begin{align}\label{eq:prop:Holder:time:I1}
	(I_1)_{t'}(x) - (I_1)_t(x) = \sum_{x_1} \bfp_{t'-t}(x-x_1)((I_1)_t(x_1)-(I_1)_t(x)).
\end{align}
By \eqref{eq:prop:Holder:I1:del}, we have
\begin{align}\label{eq:prop:Holder:time:I1:1}
	\VertBK{(I_1)_t(x_1)-(I_1)_t(x)}_{2j}
\leq
	\BK{\E|x-x_1|}^{u/2} e^{a|x-x_1|} e^{2a|x|} 
	C N_{a,2j,u/2}.
\end{align}
Combining \eqref{eq:prop:Holder:time:I1:1} with \eqref{eq:p:esti:sum} and \eqref{eq:prop:Holder:time:I1},
we conclude \eqref{eq:prop:Holder:time} for $i=1$.

For $i=2$, write $(I_2)_{t'}-(I_2)_t$ as the sum of 
$I_{21}:=\int_t^{t'} \bfp_{t'-s}*\ZE_sd\ME_s$
and $I_{22}:=\int_0^t (\bfp_{t'-s}-\bfp_{t-s})*\ZE_sd\ME_s$.
Similar to \eqref{eq:prop:Holder:I2:abs},
applying Lemma \ref{lem:QVcomparision}, \eqref{eq:p:esti:sup},
\eqref{eq:p:esti:Linfty}, and \eqref{eq:p:esti:sum} to $I_{21}$,
we bound $\Vert I_{21}\Vert_{2j}^2$ by 
\begin{align*}
	C \E \int_t^{t'} (t'-s)^{-\frac12} e^{2a\E|x|} \absBK{(\ZE_s)^2}_{2a,j} ds.
\end{align*} 
By \eqref{eq:Z:moment} and $t'-t \le \E^{-2}\tilT$,
we further bound this integral by 
$
	C \sqBK{ \E^\frac{u}{2}(t'-t)^\frac{u}{4} e^{2a\E|x|} |\ZE_0|_{a,2j}  }^2.
$
Similarly, applying  Lemma \ref{lem:QVcomparision} and \eqref{eq:p:esti:sup} 
to $I_{22}$,
using 
$
	(\bfp_{t'-s}-\bfp_{t-s})^2 
	\le
	|\bfp_{t'-s}-\bfp_{t-s}| (\bfp_{t'-s}+\bfp_{t-s}), 
$
and then combining \eqref{eq:p:holder:time:esti} for $v=\frac{u}{2}$,
\eqref{eq:p:esti}, and \eqref{eq:Z:moment},
we bound $\Vert I_{21}\Vert_{2j}^2$ by 
$C\sqBK{ \E^\frac{u}{2}(t'-t)^\frac{u}{4} e^{2a\E|x|} |\ZE_0|_{a,2j}}^2$.
The estimates for $\Vert I_{21}\Vert_{2j}^2$ and $\Vert I_{21}\Vert_{2j}^2$
conclude \eqref{eq:prop:Holder:time} for $i=2$.

For $i=3$, write $(I_2)_{t'}-(I_2)_t$ as the sum of 
$I_{31}:=\int_t^{t'} \E^2 \bfp_{t'-s}*\calE_s\ZE_s ds$
and $I_{32}:=\E^2\int_0^t (\bfp_{t'-s}-\bfp_{t-s})*\calE_s\ZE_s ds$.
Applying \eqref{eq:p:esti:sum} and \eqref{eq:Z:moment},
we bound $\Vert I_{31} \Vert_{2j}$ by a constant multiple of
\begin{align*}
	\E^2(t'-t) e^{a\E|x|} |\ZE_0|_{a,2j}
\leq 
	\tilT^{1-\frac{u}{4}} 
	\E^\frac{u}{2}(t'-t)^\frac{u}{4} e^{a\E|x|} |\ZE_0|_{a,2j}.
\end{align*}
For $I_{32}$,
similar to \eqref{eq:prop:Holder:I34:CS},
by introducing a weight $g_{s}(x):=e^{|x|(1\wedge s^{-1/2})}$ 
in the Cauchy--Schwartz inequality,
we bound $\Vert I_{32}\Vert_{2j}^2$
by  
\begin{align}\label{eq:prop:Holder:5}
	C
	\int_0^t \E^2 
	\sqBK{ \BK{|\bfp_{t'-s}-\bfp_{t-s}|/g_{t-s} } * 1 } ds 
	\
		\int_0^t \E^2 \BK{(\bfp_{t'-s}+\bfp_{t-s})g_{t-s}}
		*\VertBK{(\ZE_s)^2}_j ds.
\end{align}
For the first integral, using \eqref{eq:p:holder:time:esti} for $v=\frac{u}{2}$
we bound it by $C \E^u (t'-t)^\frac{u}{2}$.
For the second integral,
combing \eqref{eq:p:esti:sum} and \eqref{eq:Z:moment}
we bound it by $C e^{2a\E|x|} |\ZE_0|_{a,2j}^2$.
Hence $\Vert I_{32}\Vert_{2j}^2$ is bounded by 
$C[\E^\frac{u}{2}(t'-t)^\frac{u}{4} e^{a\E|x|} |\ZE_0|_{a,2j}]^2$.
Our estimates for $I_{31}$ and $I_{32}$ conclude 
\eqref{eq:prop:Holder:time} for $i=3$.

Finally, for $i=4$, we similarly define
$I_{41}:=\int_t^{t'} \E \nabla_k\bfp_{t'-s}*\calE^k_s\ZE_s ds$
and $I_{42}:=\int_0^t (\nabla_k\bfp_{t'-s}-\nabla_k\bfp_{t-s})*\calE_s^k\ZE_s ds$.
For $I_{41}$, 
applying \eqref{eq:del:p:esti:sum} for $v=1$ and \eqref{eq:Z:moment} 
we bound $\Vert I_{41} \Vert_{2j}$ by a constant multiple of
\begin{align*}
	\E(t'-t)^\frac12 e^{a\E|x|} |\ZE_0|_{a,2j}
\leq 
	\tilT^{\frac12-\frac{u}{4}} 
	\E^\frac{u}{2}(t'-t)^\frac{u}{4} e^{a\E|x|} |\ZE_0|_{a,2j}.
\end{align*}
For $I_{42}$,
similar to \eqref{eq:prop:Holder:5},
we bound $\Vert I_{42}\Vert_{2j}^2$
by a constant multiple of 
\begin{align*}
	\sum_{|k|\le m} 
	\int_0^t \E 
	\BK{|\nabla_k\bfp_{t'-s}-\nabla_k\bfp_{t-s}|/g_{t-s} } * 1 ds 
	\
	\int_0^t \E \BK{(\nabla_k\bfp_{t'-s}+\nabla_k\bfp_{t-s})g_{t-s}}
	*\VertBK{(\ZE_s)^2}_j ds.
\end{align*}
For the first integral, using 
\eqref{eq:del:p:holder:time:esti} for $v=\frac{u}{2}$
we bound it by $C \E^u (t'-t)^\frac{u}{2}$.
For the second integral,
combing \eqref{eq:del:p:esti:sum} and \eqref{eq:Z:moment}
we bound it by $Ce^{2a\E|x|}|\ZE_0|_{a,2j}^2$.
Hence $\Vert I_{42}\Vert_{2j}^2$ is bounded by
$C [\E^\frac{u}{4}(t'-t)^\frac{u}{2} e^{a\E|x|} |\ZE_0|_{a,2j}]^2$.
Our estimates for $I_{41}$ and $I_{42}$ conclude 
\eqref{eq:prop:Holder:time} for $i=4$.
\end{proof}

Proposition \ref{prop:Holder} immediately implies following corollary
\begin{corollary}\label{cor:Holder}
Under the assumptions \eqref{eq:initial:cond:bd} and \eqref{eq:initial:cond:Holder},
for all $u\in(0,1)$ we have the following estimates
\begin{align}
	\label{eq:moment:esti}
	\VertBK{\ZE_t}_{14} &\le C(u) e^{a_0\E|x|},
\\
	\label{eq:Holder:esti:space}
	\VertBK{\ZE_t(x)-\ZE_t(x')}_{14} &\leq
	C(u) (\E|x-x'|)^\frac{u}{2} e^{a_0(\E|x|+|x'|)},
\\
	\label{eq:Holder:esti:time}
	\VertBK{ \ZE_t(x) - \ZE_{t'}(x) }_{14} &\leq
	C(u) (1\vee|t'-t|^\frac{u}{4}) \E^\frac{u}{2} e^{2a_0\E|x|},
\end{align}
where $a_0$ is the same as in \eqref{eq:initial:cond:bd} and \eqref{eq:initial:cond:Holder}.
\end{corollary}

\begin{proof}[Proof of Proposition \ref{prop:main:tight}]

This proposition is the generalization of 
the first half (tightness) of \cite[Theorem 3.3]{bertini97}
to $m>1$.
The original proof by \cite{bertini97} for $m=1$
actually applies to all processes satisfying the conclusions of
\cite[Lemma 4.1, 4.5-4.7]{bertini97},
whose proofs relies only on the conclusions of \cite[Lemma 4.1-4.3]{bertini97}
for $p>12$ and the fact that $N_i(x')$, as defined in \eqref{eq:Qiy:defn},
is stochastically bounded by a Poisson random variable with a fixed rate.
(Specifically, the assumption $p>12$ is used in \cite[(4.60)]{bertini97}.)
In our case, $\ZE$ satisfies  
\eqref{eq:moment:esti}, \eqref{eq:Holder:esti:space}, and \eqref{eq:Holder:esti:time},
which correspond to the conclusions of \cite[Lemma 4.1, 4.2, 4.3]{bertini97}
for $p=14>12$, respectively,
and $N_i(y)$ is stochastically bounded by a Poisson random variable with rate
$(\sum_{k=1}^m r_k k)$, as shown in the proof of Lemma \ref{lem:QVcomparision}.
\end{proof}

\subsection{Convergence}

In this section we prove Proposition \ref{prop:main:SHE}.
To this end, we first obtain an expression of the 
predictable quadratic variation of $\ME$.

\begin{proposition}\label{prop:QVexpression}
$d\angleBK{\ME(x),\ME(x')}$ vanishes unless $|x-x'|<m$,
and for $|x-x'|<m$
\begin{align*}
	 \ZE_t(x) \ZE_t(x')
	d\angleBK{\ME(x),\ME(x')}
=
	\E \BK{ \lambda^2 \alpha_{l} + \calE_t(x) } \ZE_t(x)^2 dt,
\end{align*}
where $l:=x-x'$ and
\begin{align}\label{eq:alphal:defn}
	\alpha_{l} := \sum_{k=|l|+1}^m r_k(k-|l|).
\end{align}
\end{proposition}

\begin{proof}
By the independence of the Poisson processes $Q^k_t(y)$, we have
\begin{align}\label{eq:lem:QVexpression:Poiss}
 	\angleBK{Q^{y_2-y_1}_t(y_1),Q^{y_2'-y_1'}_t(y_1')}_t
=
	\int_0^t 
 	\bbbone_\curBK{(y_1,y_2)=(y'_1,y'_2)}\qE_{y_2-y_1} ds.
\end{align}
Rewrite \eqref{eq:M:defn} as
\begin{equation}\label{eq:M:expression}
	 d\ME_t(x) = 
	\sum_{y_1,y_2}
	\BK{ e^{2\sign(y_2-y_1)\lambdaE} -1 } 
	 a_{y_1\to y_2}(\eta_t)
	\BK{ dQ^{y_2-y_1}_t(y_1) - \qE_{y_2-y_1} dt },
\end{equation}
where the sum is taken over all hops $y_1\to y_2$ across $x$.
We then take the product of $d\ME(x)$ and $d\ME(x')$ 
using \eqref{eq:lem:QVexpression:Poiss} and \eqref{eq:M:expression} to get
\begin{align*}
	d\angleBK{\ME(x),\ME(x')}
=
	\sum_{y_1,y_2}
	(e^{2\sign(y_2-y_1)\lambdaE}-1)^2 a_{y_1\to y_2}(\eta_t) \qE_{y_2-y_1}dt.
\end{align*}
Here the sum is taken over all hops $y_1\to y_2$ that cross both $x$ and $x'$,
hence is nonzero only when $|x-x'|<m$,
where by putting $(y_1,y_2)=(x\pm \barj,x\pm \barj\mp k)$ we get 
\begin{align}\label{eq:lem:QVexpression:1}
\begin{split}
	&
	d\angleBK{\ME(x),\ME(x')}
	=
	\sum_{k=l+1}^m \sum_{j=1\vee(1+l)}^{k\wedge(k+l)}
	\Big( 
		(e^{-2\lambdaE}-1)^2 a_{x+\barj \to x+\barj-k}(\eta_t) \qE_{-k} 
\\
		&\quad	
		+ (e^{2\lambdaE}-1)^2 a_{x-\barj\to x-\barj+k}(\eta_t) \qE_{k} 
	\Big) dt.
\end{split}
\end{align}
Taylor-expanding $(e^{\pm 2\lambdaE}-1)^2$ in \eqref{eq:lem:QVexpression:1} to the first order
and using \eqref{eq:a:indicator:defn},
we obtain
\begin{align}\label{eq:lem:QVexpression:2}
\begin{split}
	&
	d\angleBK{\ME(x),\ME(x')}
	=
	\E\sum_{k=|l|+1}^m \sum_{j=1\vee(1+l)}^{k\wedge(k+l)}
	\Big( \lambda^2 (\qE_{k}+\qE_{-k} + w_1) + w_2 \Big) dt,
\end{split}
\end{align}
where $w_1$ is a sum of $\eta$-linear and $\eta$-quadratic terms 
and $w_2$ is uniformly vanishing.
By \eqref{eq:Z:defn} and Taylor expansion to the first order,
we have 
$
	\ZE_t(x') = \ZE_t(x)\BK{ 1 + w_3 },
$
for some uniformly vanishing $w_3$.
Clearly $w_1w_3$ and $w_2w_3$ are uniformly vanishing,
and by Lemma \ref{lem:pre:replace} $w_1$ is weakly vanishing.
Multiplying \eqref{eq:lem:QVexpression:2} by 
$
	\ZE_t(x)\ZE_t(x')
=
	\ZE_t(x)^2 \BK{ 1 + w_3 }
$
and using $\qE_k+\qE_{-k}=r_k$, we conclude the proof.
\end{proof}

We next use use a  martingale problem to prove Proposition \ref{prop:main:SHE}.

\begin{definition}\label{defn:MGprob}
Let $Z_\cdot\dotBK$ be a $C([0,\infty),C(\bbR))$-valued process such that
given any $\tilT>0$, there exists $A\ge 0$ such that
\begin{equation}\label{eq:MGprob:initial}
	\sup_{T\in[0,\tilT]} \sup_X
	e^{-A|X|} \bbE \BK{Z_T(X)^2}
<
	\infty.
\end{equation}
The process $Z_\cdot\dotBK$ solves the martingale problem with initial condition $\calZ_0$ 
if $Z_0=\calZ_0$ in distribution and
\begin{align*}
	(\phi,Z_T) - (\phi,Z_0)
	- \frac{1}{2} \int_0^T \BK{\phi'',Z_S} dS,
	\quad
	 N_T(\phi)^2
	 - \int_0^T (\phi^2,Z_S^2) dS,
\end{align*}
are local martingale for any $\phi\in C^\infty_c(\bbR)$,
where $(\phi,\psi):= \int_\bbR \phi(x)\psi(x) dx$.
\end{definition}

\begin{proof}[Proof of Proposition \ref{prop:main:SHE}]

Recall from \cite[Proposition 4.11]{bertini97} that
for any initial condition $\calZ_0$ satisfying 
\begin{align}\label{eq:prop:main:SHE:ic}
	\VertBK{\calZ_0(X)}_{2}
	\leq C e^{A|X|},
	\text{ for some } A >0,
\end{align}
the martingale problem Definition \ref{defn:MGprob} has a unique solution,
which coincides with the law of the solution to \eqref{eq:SHE:mild} with initial condition $\calZ_0$.
Consequently, it suffices to show that \eqref{eq:prop:main:SHE:ic}
holds and any limit point $\calZ$ of $\curBK{\calZE}$ solve the martingale problem
Definition \ref{defn:MGprob} starting from $Z_0$.
By passing to the relative subsequence we assume
$\calZE_\cdot\dotBK\Rightarrow \calZ_\cdot\dotBK$.

Clearly, \eqref{eq:prop:main:SHE:ic} and \eqref{eq:MGprob:initial} 
hold because of \eqref{eq:initial:cond:bd} and \eqref{eq:moment:esti}, respectively.
Let $\usZ_T(X):=\ZE_{\E^{-2}T}(\E^{-1}X)=\calZE_{\beta^{-1}T}((\beta')^{-1}X)$.
By the change of variables $(T,X)\mapsto(\beta^{-1}T,(\beta')^{-1}X)$, 
it suffices to show that
\begin{align}
	\label{eq:MGprob:N}
	N_T(\phi)
&:=
	(\phi,\usZ_T) - (\phi,\usZ_0)
	- \frac{\alpha}{2} \int_0^T \BK{\phi'',\usZ_S} dS,
	\\
	\label{eq:MGprob:Lambda}
	\Lambda_T(\phi)
&:=
	 N_T(\phi)^2
	 - \alpha\lambda^2 \int_0^T (\phi^2,\usZ_S^2) dS,
\end{align}
are local martingales.
Let $a$ be an arbitrary positive number, $0\le S \le S'\le\tilT$,
and let $f:D([0,S],C(\bbR))\to\bbR$ be bounded and continuous with respect to the Skorokhod topology.
Associate with any process $A_T$ with the stopped process
\begin{align*}
	A^a_T:= A^a_{T\wedge T_a},
	\quad
	T_a := \inf\curBK{T\geq 0: |A_T|>a}.
\end{align*}
It suffices to prove
\begin{equation*}
	\bbE\BK{ (N_{S'}(\phi)^a - N_S(\phi)^a) f(N_\cdot) } =0,
\quad
	\bbE\BK{ (\Lambda_{S'}(\phi)^a - \Lambda_S(\phi)^a) f(N_\cdot) }
	=0.
\end{equation*}
To this end, set $(\phi,\ZE_t)_\E:=\E\sum_x \phi(\E X)\ZE_t(x)$, 
$\bfL^\E:=\sum_{k=1}^m \tilrE_k \Delta_k$, and 
\begin{align*}
	\NE_T(\phi)
	&:=
	(\phi, \ZE_{\E^{-2}T})_\E - (\phi,\ZE_0)_\E
	-
	\frac{1}{2} \int_0^{\E^{-2}T} \BK{ \bfL^\E\phi, \ZE_s }_\E ds,
\\
	\LambdaE_T(\phi)
	&:=
	\NE_T(\phi)^2
	- \E^2 \alpha\lambda^2 \int_0^{\E^{-2}T} \BK{ \phi^2, (\ZE_s)^2 }_\E ds. 
\end{align*}
By the definition of weak convergence,
using $\sum_{k=1}^m k^2\tilrE_k\to \alpha$ (by \eqref{eq:tilr:approx} and \eqref{eq:alpha:defn}) and 
using localization by $T_a$ to guarantee the boundedness of
$\NE_S(\phi)^a$ and $\LambdaE_S(\phi)^a$, we get
\begin{align}\label{eq:prop:main:SHE:0}
\begin{split}
	\bbE\BK{ (N_{S'}(\phi)^a - N_S(\phi)^a) f(N_\cdot) }
	&=
	\lim_{\E\to 0}
	\bbE\BK{ (\NE_{S'}(\phi)^a - \NE_S(\phi)^a) f(\NE_\cdot) },
\\
	\bbE\BK{ (\Lambda_{S'}(\phi)^a - \Lambda_S(\phi)^a) f(N_\cdot) }
	&=
	\lim_{\E\to 0}
	\bbE\BK{ (\LambdaE_{S'}(\phi)^a - \LambdaE_S(\phi)^a) f(\NE_\cdot) }.
\end{split}
\end{align}

We next show that $\NE_{S}(\phi)$ and $\LambdaE_S(\phi)$ are approximated by
the martingales
\begin{align*}
	\tilNE_t :=
	\int_0^t \BK{\phi,\ZE_s d\ME_s}_\E ds,
\quad
	(\tilNE_t)^2 - \angleBK{\tilNE}_t,
\end{align*}
respectively, and that the RHS of \eqref{eq:prop:main:SHE:0} are zero.
To this end, we integrate \eqref{eq:discr:SHE} to get
\begin{equation}\label{eq:prop:main:SHE:0:1}
	\NE_{\E^{-2}t}(\phi) 
	- \tilNE_t
	=
	\E^2 \int_0^t \BK{\phi,\calE_s}_\E ds 
	+ \sum_{|k|\le m} \E \int_0^t \BK{\nabla_k\phi,\calE^k_s}_\E ds.
\end{equation}
Let  $\NE_1$ and $\NE_2$ denote the 
first and second terms on the RHS of \eqref{eq:prop:main:SHE:0:1}, respectively.
By Taylor-expanding $\nabla_k\phi$ to the first order, we rewrite $\NE_2$ as the sum of
\begin{align*}
 	\NE_{21} 
 	:= 
 	\sum_{|k|\le m} \E^2 \int_0^t \BK{k\phi',\calE^k_s\ZE_s}_\E ds
\quad
\text{and}
\quad
 	\NE_{22} 
 	:= 
 	\E^3 \int_0^t \sum_{|k|\le m} \BK{\phiE_k,\calE^k_s\ZE_s}_\E ds,	
\end{align*}
where $\phiE_k(x)$ is a bounded (in $x$ and $\E$) function.
By Proposition \ref{prop:QVexpression},
\begin{align*}
	\angleBK{\tilNE}_t =
	\E^2 \int_0^t \sum_x \sum_{|l|<m} \phi(x) \phi(x+l)
	\lambda^2 \BK{\alpha_l + \calE_s(x)} \ZE_s(x)^2 ds.
\end{align*}
By $\phi(x)\phi(x+l)=\phi(x)^2 + O(\E l)$
and $\sum_{|l|<m} \alpha_l = \alpha$ (by \eqref{eq:alpha:defn} and \eqref{eq:alphal:defn}),
we further write $\tilNE_t$ as the sum of
\begin{align*}
	\tilLambdaE &:=
	\alpha \lambda^2 \E^2 \int_0^t \BK{ \phi^2, (\ZE_s)^2 } ds,
\quad
	\LambdaE_1 := 
	\E^2 \int_0^t \BK{ \phi^2, \calE_s\ZE_s}_\E ds,
\\
	\LambdaE_2 &:= 
	\E^3 \int_0^t 
	\BK{ \psiE, \calE_s\ZE_s}_\E ds,
\end{align*}
where $\psiE(x)$ is a bounded (in $x$ and $\E$).
Hence
\begin{align}\label{eq:prop:main:SHE:2}
	\LambdaE_{\E^2t}(\phi)
	= 
	\BK{ \tilNE_t + (\NE_1)_t + (\NE_2)_t }^2 - \tilLambdaE_t
	=
	(\tilNE_t)^2 - \langle\tilNE\rangle_t + \RE_t,
\end{align}
where
\begin{align*}
	\RE_t:=
	 (\LambdaE_1)_t + (\LambdaE_2)_t
	+ 2 \BK{ (\NE_1)_t+(\NE_2)_t } \tilNE_t
	+ \BK{ (\NE_1)_t+(\NE_2)_t }^2.
\end{align*}
The proof is completed upon showing
\begin{align}\label{eq:prop:main:SHE:5}
	\lim_{\E\to 0} \Vert (\NE_1)_t\Vert_2 =0,
	\quad
	\lim_{\E\to 0} \Vert(\NE_2)_t\Vert_2 = 0,
	\quad
	\lim_{\E\to 0} \Vert \RE_t\Vert_1 = 0.
\end{align}
Since $\phi^\E$ and $\psi^\E$ are bounded, 
by \eqref{eq:moment:esti} we have
\begin{align}
	&
	\label{eq:prop:main:SHE:1}
	\Vert (\NE_1)_t\Vert_4, \
	\Vert((\NE_{21})_t\Vert_4, \
	\Vert (\LambdaE_1)_t\Vert_4
	\le C,
\\
	&
	\label{eq:prop:main:SHE:3}
	\Vert (\tilNE_t)^2 \Vert_2 \le C,
	\quad
	\Vert (\NE_{22})_t \Vert_2 
	\le C \E,
	\quad
	\Vert (\LambdaE_{2})_t \Vert_2
	\le C \E.
\end{align}
The bound \eqref{eq:prop:main:SHE:1} implies the uniformly integrability of
$\{(\NE_1)_t^2\}_\E$, $\{(\NE_{21})_t^2\}_\E$, and $\{(\LambdaE_1)^2_t\}_\E$, 
which together with Lemma \ref{lem:pre:replace}
implies 
\begin{align}\label{eq:prop:main:SHE:4}
	\lim_{\E\to 0} \Vert (\NE_1)_t\Vert_2 =0,
	\quad
	\lim_{\E\to 0} \Vert(\NE_{21})_t\Vert_2 = 0,
	\quad
	\lim_{\E\to 0} \Vert (\LambdaE_1)_t \Vert_1 = 0.
\end{align}
Combining \eqref{eq:prop:main:SHE:3} and \eqref{eq:prop:main:SHE:4} we prove \eqref{eq:prop:main:SHE:5}.
\end{proof}

\section{Replacement lemma}
\label{sec:replace:lem}
We first recall some basic notions of continuous time Markov processes associated with exclusion processes.

As mentioned in Remark \ref{rmk:vanishing},
for any $a\in[0,1]$, the product measure $\nu_a$ 
is an invariant measure of exclusion process.
Let $\calD:=\curBK{\pm 1}^\halfZ$ be quipped with the corresponding cylindrical $\sigma$-algebra $\calG_\infty$
and the probability measure $\nu:=\nu_\frac12$,
let $\Lambda_n:=(-n,n)\cap\halfZ$ be the $n$-th interval around $0$,
and let $\{\calG_n\}$ be the filtration corresponding to the restriction to $\curBK{\pm 1}^{\Lambda_n}$ 
of functions $g:\calD\to\bbR$.  
For a function $g:\calD\to\bbR$ define 
\begin{align*}
	(\sigma_{y,y+k}g)(\eta) := g(\eta^{y,y+k}),
\quad
	\eta^{y_1,y_2}(y)
	:=
	\left\{\begin{array}{l@{,}l}
		\eta(y_2)	&\text{ when } y=y_1,	\\
		\eta(y_1)	&\text{ when } y=y_2, \\
		\eta(y)		&\text{ otherwise,}
	\end{array}\right.
\end{align*}
Recall from \eqref{eq:a:indicator:defn}, $a_{y_1\to y_2}$ is the indicator function for allowed hops. Let
\begin{align}
	\label{eq:c:defn}
	\cE_{y_1,y_2} &:= 
	\BK{ \qE_{y_2-y_1} a_{y_1\to y_2} + \qE_{y_1-y_2} a_{y_2\to y_1} } \bbbone_\curBK{0<|y_1-y_2|\le m},
\\
	\label{eq:c:infty:defn}
	c^\infty_{y_1,y_2} &:=
	\lim_{\E\to 0} \cE_{y_1,y_2}
	= r_{|y_1-y_2|} \BK{ a_{y_1\to y_2} + a_{y_2\to y_1} } \bbbone_\curBK{0<|y_1-y_2|\le m}.
\end{align}
The Markov generator of the exclusion process is 
\begin{align}
	\label{eq:L:gentr:defn}
	\LE &:= \sum_{y_1<y_2} \LE_{y_1,y_2},
	\quad
	\LE_{y_1,y_2} g := \cE_{y_1,y_2} \BK{ (\sigma_{y_1,y_2}g) - g }.
\end{align}
Let $\mu_\tE$ denote the law of the exclusion process on $\calD$ at time $t$, 
and $f_\tE:=\frac{d\mu_\tE}{d\nu}$.
Recall that a function $g_t(\eta):\calD\times[0,\infty)\to \bbR$ (respectively $g:\calD\to\bbR$ ) is cylinder 
if there exists $n$ such that for all $t$, $g_t\in\calG_n$ (respectively $g\in\calG_n$).
By the forward Kolmogorov equation, for any cylinder $g_t$,
\begin{equation}\label{eq:Kolm:eq}
	\frac{d~}{dt} \bbE_\nu \BK{f_\tE g_t} =
	\bbE_\nu \BK{f_\tE \, \partial_tg_t} 
	+ \bbE_\nu \BK{f_\tE \, \LE g_t}.
\end{equation}
For cylinder $g$, define
\begin{align}
	\label{eq:Diri:Dyy:defn}
	\DE_{y_1,y_2}(g) 
	&:= 
	\frac12 \bbE_\nu \BK{ \cE_{y_1,y_2} ((\sigma_{y_1,y_2}g)^\frac12 - g^\frac12)^2 },
\\
	\label{eq:Diri:infty:Dyy:defn}
	D^\infty_{y_1,y_2}(g) 
	&:= 
	\frac12 \bbE_\nu \BK{ c^\infty_{y_1,y_2} ((\sigma_{y_1,y_2}g)^\frac12 - g^\frac12)^2 },
\\
	\label{eq:Diri:defn}
	\DE(g)
	&:= 
	\sum_{y_1<y_2} \DE_{y_1,y_2}(g).
\end{align}
Note that the sum \eqref{eq:Diri:defn} is finite since $g$ is cylinder.
For each $y_1,y_2$, the Dirichlet forms \eqref{eq:Diri:Dyy:defn} and \eqref{eq:Diri:infty:Dyy:defn}
are a convex and lower-semicontinuous function of $g$ 
(see \cite[Theorem A.1.10.2]{kipnis99} and \cite[Corollary A.1.10.3]{kipnis99}).
We have the identity (see \cite[Theorem A.1.9.]{kipnis99})
\begin{equation}\label{eq:Dirichelet:formula}
	\bbE_\nu \BK{ g^\frac12 \, \LE g^\frac12 } 
	= -\DE(g).
\end{equation}
For $0<|y_1-y_2|\le m$,
since $c^\infty_{y_1,y_2}+c^\infty_{y_2,y_1}=r_{|y_1-y_2|}(a_{y_1\to y_2}+a_{y_2\to y_1})$
and $\sigma_{y_1,y_2}g=g$ unless the hop $y_1\to y_2$ or the hop $y_2\to y_1$ is allowed,
we have
\begin{align}\label{eq:Diri:infty:Dyy:variant}
	D^\infty_{y_1,y_2}(g)  
	= 
	r_{|y_1-y_2|} \bbE_\nu \BK{((\sigma_{y_1,y_2}g)^\frac12 - g^\frac12)^2}.
\end{align}
Since $\lim_{\E\to 0} \qE_k = r_{|k|}>0$, for all $\E$ small enough we have 
\begin{align}\label{eq:Diri<Diri}
	2 D^\infty_{y_1,y_2}(g) \le \DE_{y_1,y_2}(g).
\end{align}
Let $g^n:= \bbE_\nu(g|\scrG_n)$.
For a probability density function $g$ define the $n$-th entropy as 
$H^n_t(g):=\bbE_\nu (g^{n} \log(g^n))$.
Let $\calD_n:=\curBK{\pm 1}^{\Lambda_n}$.
Since $\#\calD_n=2^{2n}$, we have the crude bounds
\begin{align}
	\label{eq:H:crude:bound}
	H^n_t(g) &\le 2n \log 2,
\\
	\label{eq:D:curde:bound}
	\DE(g^n) &\le C e^{nC}.
\end{align}

Next we gives a bound on the Dirichlet form:

\begin{lemma}\label{lem:Dirichlet:form}
We have the estimate 
\begin{align}\label{eq:Dirichlet:esti}
	\DE \BK{ t^{-1}\int_0^t f^n_{t,\E} ds } \le C + Ct^{-1} n.
\end{align}
\end{lemma}

\begin{remark}\label{rmk:replace}
For exclusion processes on the torus $\bbZ/N\bbZ$,
one get the bound $C t^{-1}n$ instead of \eqref{eq:Dirichlet:esti}.
We get the extra constant from the boundary effect of $\Lambda_n$.
Since $t^{-1}n=O(\E)$ under the scaling $n=\E^{-1}R$, $t=\E^{-2}T$,
this extra constant aggravates \eqref{eq:Dirichlet:esti}.
Consequently, in Lemma \ref{lem:replace}
the radius of averaging we obtain is of the mesoscopic scale $\E^{-\frac12}$,
not the macroscopic scale $\E^{-1}$,
which is the scale of standard replacement lemmas.
While a more careful analysis might remove the extra constant,
\eqref{eq:Dirichlet:esti} and Lemma \ref{lem:replace} 
suffice for proving Lemma \ref{lem:pre:replace}.
\end{remark}

\begin{proof}[Proof of Proposition \ref{lem:Dirichlet:form}]
Take the time derivative of $H^n_t(f_\tE)$ using \eqref{eq:Kolm:eq} to get
\begin{equation*}
	\frac{d~}{dt} H^n_t(f_\tE)
	=
	\bbE_\nu \BK{ f_\tE (\partial_t f^n_\tE/ f^n_\tE) }
	+
	\bbE_\nu \BK{ f_\tE \LE \BK{ \log f^n_\tE } }.
\end{equation*}
By the tower property of $\sigma$-algebras,
the first expectation is $\bbE_\nu \BK{ \partial_t f^n_\tE } = d1/dt =0$.
Applying the definition \eqref{eq:L:gentr:defn} of $\LE$
and the inequality $\log(b/a) \le 2a^{-\frac12}(b^\frac12-a^\frac12)$
(which holds for all non-negative $a,b$),
we obtain
\begin{align*}
	\frac{d~}{dt}H^n_t(f_\tE)
\leq
	\sum_{y_1<y_2 }
	\bbE_\nu
	\sqBK{
		2 \cE_{y_1,y_2} f^{n+m}_\tE \BK{f^n_\tE}^{-\frac12}
		\BK{ \BK{\sigma_{y_1,y_2}f^n_{t,\E}}^\frac12 - \BK{\fntE}^\frac12 }
		}.
\end{align*}
Write $f^{n+m}_\tE \BK{f^n_\tE}^{-\frac12}$ as the sum of 
$\BK{f^n_\tE}^\frac12$ and $\BK{f^{n+m}_\tE-f^n_\tE} \BK{f^n_\tE}^{-\frac12}$.
Using \eqref{eq:Dirichelet:formula}, we have
\begin{equation}\label{eq:lem:Dirichlet:form:1}
	\frac{d~}{dt}H^n_t(f_\tE) 
	\leq
	- 2 \DE(f^n_\tE) + D',
\end{equation}
where
\begin{align}\label{eq:lem:Dirichlet:form:D'}
	D'
:=
	2\sum_{y_1<y_2}
	\bbE_\nu \sqBK{
		\cE_{y_1,y_2} \BK{f^{n+m}_\tE-f^n_\tE} \BK{f^n_\tE}^{-\frac12} 
		\BK{ \BK{\sigma_{y_1,y_2}f^n_\tE}^\frac12 - \BK{f^n_\tE}^\frac12 }
	}.
\end{align}
Note that in \eqref{eq:lem:Dirichlet:form:D'} we need only to sum over
\begin{align}
	\partial\Lambda_n := \curBK{(y_1,y_2): y_1<y_2, \, |y_1-y_2|\le m, \text{ exactly one of } y_1,y_2 \in\Lambda_n }.
\end{align}
Indeed, when $y_1,y_2\in(\Lambda_n)^c$, we have $\sigma_{y_1,y_2}f^n_\tE - f^n_\tE=0$,
and when $y_1,y_2\in\Lambda_n$, by the tower property of $\sigma$-algebras 
we can replace $f^{n+m}_\tE$ in \eqref{eq:lem:Dirichlet:form:D'} by $f^{n}_\tE$.
Applying the inequality $ab\le (a^2R^{-1}+Rb^2)2^{-1}$, we further obtain, for any $R>0$,
\begin{equation}\label{eq:lem:Dirichlet:form:D':1}
	D'
\leq
	\frac{1}{2R} \sum_{(y_1,y_2)\in\partial\Lambda_n} \DE_{y_1,y_2}(f_\tE)
	+ D''
\leq
	\frac{1}{2R} \DE(f^{n+m}_\tE) + D'',
\end{equation}
where
\begin{align}\label{eq:lem:Dirichlet:form:D''}
	D'':=
	\frac{R}{2}
	\sum_{(y_1,y_2)\in\partial\Lambda_n} 
	\bbE_\nu \sqBK{ \cE_{y_1,y_2} 
	\BK{(\sigma_{y_1,y_2}f^n_\tE)^\frac12-(f^n_\tE)^\frac12}^2{f^n_\tE}^{-1} }.
\end{align}
Since $|c_{y_1,y_2}|\le 1$ and $(a-b)^2 b^{-1} \le 2(a^2+b^2)b^{-1}$,
the random variable in \eqref{eq:lem:Dirichlet:form:D''} is bounded by 
$2[(f^{n+m}_\tE/f^n_\tE)f^{n+m}_\tE + f^n_\tE]$.
Since $1=\bbE_\nu \sqBK{ \,f^{n+m}_\tE\!\big/f^n_\tE \, \middle|\scrG_n} \,(\eta)$ per $\eta$ 
is the equally weighted average of the $2^{2m}$ values that $\fnmtE/\fntE$ can take, 
it follows that $ \fnmtE/\fntE \le 2^{2m}$.
Hence $D'' \le C R (\# \partial\Lambda_n) \le C(R)$.
Combining this with \eqref{eq:lem:Dirichlet:form:1} and \eqref{eq:lem:Dirichlet:form:D':1},
we obtain 
\begin{equation}\label{eq:lem:Dirichlet:form:2}
	\frac{d~}{dt} H^n_t(f_\tE)
	+ 2\DE(\fntE) 
\leq
	\frac{1}{2R} \DE(f^{n+m}_\tE) + C(R).
\end{equation} 
Consider \eqref{eq:lem:Dirichlet:form:2} for $n=n+j$, $j\in\{0,1,\ldots\}$.
Integrating in time, multiplying by $e^{-bj}$, and summing over $j\in\{0,1,\ldots\}$,
we obtain
\begin{equation}\label{eq:lem:Dirichlet:form:3}
	- \sum_{j=0}^\infty e^{-bj} H^n_t(f_\tE) 
	+ 2 D'''_0
\leq
	\frac{e^{mb}}{2R} D'''_m + C(R) t,
\end{equation} 
where
\begin{align}\label{eq:lem:Dirichlet:form:4}
	D'''_k := \sum_{j=k}^\infty e^{-bj}\int_0^t \DE(f^{n+j}_\sE) ds.
\end{align}
Note that by \eqref{eq:H:crude:bound} and \eqref{eq:D:curde:bound},
the sums of \eqref{eq:lem:Dirichlet:form:3} and \eqref{eq:lem:Dirichlet:form:4} 
are finite for all large enough $b$.
Fix one such $b$, choose $R$ so that $e^{mb}(2R)^{-1} \le 1$,
yielding $2D'''_0 - e^{mb}(2R)^{-1} D'''_m \le D'''_0$.
By \eqref{eq:H:crude:bound} we have $\sum_{j=0}^\infty e^{-bj} H^n_t(f_\tE) \le C n$.
Thus, we obtain
\begin{equation*}
	\int_0^t \DE(f^{n}_\sE) ds 
\leq
	D'''_0
\leq
	C n + C t.
\end{equation*}
Finally, using the convexity of $\DE$ we conclude the proof.
\end{proof}

Recall that $\Psi:D\mapsto\bbR$ is Lipschitz if 
there exists $l$ and $C$ such that for all $\eta$ and $\xi$,
\begin{equation*}
	\absBK{ \Psi(\eta) - \Psi(\xi) } \le C \sum_{|y|< l} \absBK{ \eta(y) -\xi(y) }.
\end{equation*}
Let $\fint_s^t:=(t-s)^{-1}\int_s^t$ denote the average over the time interval $(s,t)$,
and for any $P\subset\bbZ$ or $P\subset \halfZ$, 
let $\avvsum_{P}:=(\# P)^{-1} \sum_P$ denote the average over $P$.
We have the readily verified inequality
\begin{align}\label{eq:double:av}
	\Bigg|
		\avsum_{|x|\le n} f(x) - \avsum_{|x'|\le j} \avsum_{|x|\le n} f(x+x')
	\Bigg|
\leq
	\frac{j}{n} \Bigg| \sum_{|x-n|\le j} f(x) + \sum_{|x+n|\le j} f(x) \Bigg|.
\end{align}
For $g:\calD\to\bbR$, put $(\tau_i g)(\eta):= g(\tau_i\eta)$, where $(\tau_i\eta)(x):=\eta(x+i)$.

Next we prove a replacement lemma that allows us 
to replace the microscopic average of $\Psi$ 
by a mesoscopic average.

\begin{lemma}\label{lem:replace}
For any Lipschitz cylinder function $\Psi$, any $T,R>0$, $T_0\geq 0$, $X_0\in\bbR$,
\begin{align}\label{eq:replace}
	\varlimsup_{\delta\to 0} \varlimsup_{\E\to 0}
	\bbE_\E \BK{ 
		\fint_{T_0}^{T_0+T} \avsum_{|x-\E^{-1}X_0|\le \E^{-1}R} 
		\tau_x V_{\deltaE}(\eta_{\E^{-2}S}) dS 
		}
	=0,
\end{align}
where
\begin{align}
	\label{eq:V:defn}
	V_n(\eta) &:= 
	\absBK{ \avsum_{|x|\le n} (\tau_x \Psi)(\eta) - \tilPsi\BK{\eta^n(0)} },
\\
	\label{eq:etal:defn}
	\eta^n(x) &:= \avsum_{|y-x|\le n} \eta(y),
\\
	\label{eq:Psi:til:defn}
	\tilPsi(a) &:= \bbE_{\nu_\frac{1+a}{2}}(\Psi),
	\text{ which is a function on } [0,1].
\end{align}
\end{lemma}

\begin{proof}

Since we impose no assumption on the initial condition, 
without lost of generality we assume $X_0=T_0=0$,
and let $\Psi$ depends only on coordinates of $\Lambda_{n_0}$.
Let
\begin{align}\label{eq:fNTE:defn}
	\fNTE:= \fint_0^{\E^{-2}T} f^{N_\E}_\tE dt,
	\quad
	N_\E:=\E^{-1}R + \delta \E^{-\frac12} + n_0.
\end{align}
We rewrite the expectation in \eqref{eq:replace} as
\begin{align}\label{eq:lem:replace:running:1}
	\bbE_\nu \Bigg[
	\avsum_{|x|\le \E^{-1}R} 
	\BK{\tau_x V_{\delta \E^{-1/2}}} \, \fNTE
	\Bigg].
\end{align}
We next use \eqref{eq:lem:replace:running:1} 
to reduce \eqref{eq:replace} to the one-block and two-blocks estimates.
In the definition \eqref{eq:V:defn} of $V_n$, add and subtract
$
	V^{(1)}
	:= 
	\avvsum_{|x|\le\delta\E^{-\frac12}} \tau_x V_l(\eta)
$
to get $V_{\deltaE} \le V^{(1)} + V^{(2)} + V^{(3)}$,
where
\begin{align}
	V^{(2)} &:= 
	\avsum_{|x|\le\delta\E^{-\frac12}} 
	\absBK{ \tilPsi(\eta^l(x)) - \tilPsi(\eta^{\delta\E^{-\frac12}}(0)) },
\\
	\label{eq:lem:replace:V3}
	V^{(3)} &:= 
	\Bigg|
		\hspace{7pt} 
		\avsum_{|x|\le \delta\E^{-\frac12}} \Psi(\tau_x\eta) 
		- \avsum_{|x|\le \delta\E^{-\frac12}}\avsum_{|x'|\le l} \Psi(\tau_{x+x'}\eta)
	\Bigg|.
\end{align}
By \cite[Corollary 2.3.6]{kipnis99} 
(which applies to any Lipschitz cylinder $\Psi$ and $\tilPsi$ as in \eqref{eq:Psi:til:defn} 
as long as $\nu_{a}\le\nu_{a'}$ for $a<a'$), $\tilPsi$ is Lipschitz,
yielding
\begin{align*}
	V^{(2)}
\leq
	C \avsum_{|x|\le\deltaE} \absBK{\eta^l(x)-\eta^{\delta\E^{-\frac12}}(0)}.
\end{align*}
By \eqref{eq:double:av}, 
we get $ V^{(3)} \le C \frac{l}{\delta\E^{-\frac12}} \sup_\eta |\Psi| \le Cl\delta^{-1}\E^\frac12$ 
(cylinder functions are bounded).
Therefore, 
\begin{align*}
	&
	\varlimsup_{\delta\to 0}\varlimsup_{\E\to 0}
	\bbE_\nu\Bigg[ 
		\avsum_{|x|\le \E^{-1}R} 
		\BK{\tau_x V_{\delta \E^{-1/2}}} \, \fNTE
	\Bigg]
\leq
	C
	\varlimsup_{l\to\infty} \varlimsup_{\E\to 0}
	\bbE_\nu\Bigg[
		\avsum_{|x|\le \E^{-1}(R+1)} \BK{\tau_x V_l} \, \fNTE
	\Bigg]
\\
	&
	\quad
	+
	C
	\varlimsup_{l\to\infty} \varlimsup_{\delta\to 0}\varlimsup_{\E\to 0}
	\sup_{|x'|\le \delta\E^{-\frac12}}
	\bbE_\nu \Bigg[ 
		\avsum_{|x|\le \E^{-1}R}
		\fNTE(\eta) 
		\absBK{ \eta^l(x+x') - \eta^{\deltaE}(x) }
	\Bigg].
\end{align*}
We thus reduce \eqref{eq:replace} to the following one-block estimate \eqref{eq:one:block} 
and two-blocks estimate \eqref{eq:two:block}.
\end{proof}

\begin{proposition}\label{prop:oneBK}
For any $R,T>0$,
\begin{align}
	\label{eq:one:block}
	&
	\varlimsup_{l\to\infty} \varlimsup_{\E\to 0}
	\bbE_\nu \Bigg[ 
		\avsum_{|x|\le \E^{-1}R} \BK{\tau_x V_l} \, \fNTE d\nu
	\Bigg]
	=0.
\end{align}
\end{proposition}

\begin{proposition}\label{prop:twoBK}
For any $R,T>0$,
\begin{align}
	\label{eq:two:block}
	\varlimsup_{l\to\infty} \varlimsup_{\delta\to 0}\varlimsup_{\E\to 0}
	\sup_{ |x'|\le \delta \E^{-\frac12} }
	\bbE_\nu \Bigg[
		\avsum_{|x|\le \E^{-1}R} 
		\fNTE
		\absBK{ \eta^l(x+x') - \eta^{\E^{-1/2}\delta}(x) }
	\Bigg]
	=0.
\end{align}
\end{proposition}

\begin{proof}[Proof of Proposition \ref{prop:oneBK}]
For any probability density function $g$ on $\calD$ define
\begin{align}\label{eq:prop:oneBK:bar:l}
	\bar g &:= \avsum_{|x|\le \E^{-1}R} \tau_x g,
\\
	\label{eq:prop:oneBK:Dirichlet}
	\Dirilinf(g) &:= \sum_{y_1,y_2\in\Lambda_l} D^\infty_{y_1,y_2}(g).
\end{align}
Note that $\overline{g^l}={\bar g}^l$,
and that $\Dirilinf$ is convex and lower-semicontinuous.
Since $\nu$ is translation invariant and $V_{l-n_0} \in\calG_{l}$,
we replace $l$ by $l-n_0$ in \eqref{eq:one:block} and rewrite the expectation as
\begin{align}\label{eq:prop:oneBK:fl}
	\bbE_\nu \sqBK{ V_{l-n_0}  \overline{ \fNTE } }
=
	\bbE_\nu \sqBK{ V_{l-n_0} \barf^l_\TE }.
\end{align}
Since the collection of all probability density functions on $\calD_l:=\curBK{\pm 1}^{\Lambda_l}$ is compact,
\begin{align}\label{eq:prop:oneBK:1}
	\varlimsup_{l\to\infty}
	\varlimsup_{\E\to 0}
	\bbE_\nu \sqBK{ V_{l-n_0} \overline{ \fNTE } }
=
	\varlimsup_{l\to\infty} 
	\bbE_\nu \sqBK{ V_{l-n_0}  f^l_\infty },
\end{align}
where $f^l_\infty := \varlimsup_{\E\to 0} \barflTE$,
which is also a probability density function on $\calD_l$.
Indeed, $\bar f^l$ is an average in expectation of $\bar f^{N_\E}$,
and by \eqref{eq:prop:oneBK:bar:l} 
$\bar f^{N_\E}$ is an average over space of $f^{N_\E}$.
Using the the convexity of $\Dirilinf$ twice, we have
\begin{align}\label{eq:prop:oneBK:Diri:conv}
	\Dirilinf(\barf^l_\TE) 
\leq
	\Dirilinf(\barf^{N_\E}_\TE)
\leq
	\avsum_{|x|\le \E^{-1}R}
	 D^\infty_l(\tau_x \fNTE).
\end{align}
Since for $\E$ small enough $N_\E > \E^{-1}R+l+m$,
by \eqref{eq:Diri<Diri} and Lemma \ref{lem:Dirichlet:form} we further get
\begin{align}\label{eq:prop:oneBK:Dirichlet:bd}
	\Dirilinf(\barflTE) \le C l \E \DE(\fNTE) \le C \E l.
\end{align}
Using \eqref{eq:prop:oneBK:Dirichlet:bd},
and the lower-semicontinuity of $\Dirilinf$, 
we obtain $\Dirilinf(f^l_\infty)=0$.

Since $r_1>0$, 
$\Dirilinf(f^l_\infty)=0$ implies $\sigma_{y_1,y_2}f^l_\infty=f^l_\infty$,
for all $y_1,y_2\in\Lambda_l$ with $|y_1-y_2|=1$.
We then deduce that $f^l_\infty$ is a function of $\sum_{|y|<l} \eta(y) = 2l \eta^l(0)$,
and since $f^l_\infty$ is a probability density function, we have
\begin{align}\label{eq:prop:oneBK:sum1}
	\sum_{|i|\le l} f^l_\infty(2i) \, 
	\bbP_\nu\BK{ \eta^l(0) = i/l } = \bbE_\nu(f^l_\infty) = 1.
\end{align}
Using \eqref{eq:prop:oneBK:sum1} we bound the expectation of \eqref{eq:prop:oneBK:1} as
\begin{align*}
	\bbE_\nu \sqBK{  V_{l-n_0} f^l_\infty }
	&=
	\sum_{|i|\leq l}  
	f^l_\infty(i) 
	\bbE_\nu \BK{  V_{l-n_0} \middle| \eta^l(0) = i/l }
	\bbP_\nu \BK{ \eta^l(0) = i/l }
\\
	&
	\quad\quad
	\leq
	\max_{|i|\le l} \ \bbE_\nu \BK{  V_{l-n_0} \middle| \eta^l(0) = i/l }.
\end{align*}
It suffices to show
\begin{align}\label{eq:prop:oneBK:2}
	\varlimsup_{l\to\infty} \max_{|i|\le l} 
	\ \bbE_\nu \BK{  V_{l-n_0} \middle| \eta^l(0) = i/l } = 0.
\end{align}

When $\eta^{l}(0)=i/l$, 
we have $\tilPsi(\eta^l(0))=\bbE_{\frac{l+i}{2l}}(\Psi):=\psi(i)$,
and by \eqref{eq:double:av} we further get,
for any $l'\le l-n_0$,
\begin{align*}
	V_{l-n_0} = \Bigg| \avsum_{|x-n_0|\le l} \tau_x \Psi + \psi(i) \Bigg|
\leq
	\avsum_{|x|\le l-n_0} \tau_x V_{l'} + C\frac{l'}{l}
\leq
	\avsum_{|x|\le l-n_0-l'} \tau_x V_{l'} + C\frac{l'}{l}.
\end{align*}
Since $\nu$ is translation invariant, 
for each $|x|\le l-l'-n_0$ we have 
$\bbE_\nu(\tau_x V_{l'}|\eta^l(0)=i/l)=\bbE_\nu(V_{l'}|\eta^l(0)=i/l)$.
Therefore, for any $l'>0$
\begin{align*}
	\varlimsup_{l\to\infty} \max_{|i|\le l} 
	\bbE_\nu \BK{  V_{l-n_0} \middle| \eta^l(0) = i/l }
\leq
	\varlimsup_{l\to\infty} \max_{|i|\le l} 
	\bbE_\nu \BK{  V_{l'} \middle| \eta^l(0) = i/l }.
\end{align*}
Indeed, $V_{l'}\in\calG_{l'+n_0}$. 
For any $\xi\in\calD_{l+n_0}$,
let $u$ be the number of its coordinates taking the value $+1$,
and let $v=2l+2n_0-u$ be the number of its coordinates taking the value $-1$.
Then we have
\begin{align*}
	&
	\bbP_\nu \BK{ \bigcap_{|y|<l'}\curBK{\eta(y)=\xi(y)} \middle| \eta^l(0)=i/l }
	= \binom{2l-u-v}{l+i-u} \left/ \binom{2l}{l+i} \right.
	= F^{u,v}_l \BK{\frac{l+i}{2l}},
\\
	&
	F^{u,v}_l(a)
	:= 
	\prod_{j_1=0}^{u-1}\BK{a-\frac{j_1}{2l}} 
	\prod_{j_2=0}^{v-1}\BK{1-a-\frac{j_2}{2l}}
	\left/\prod_{j_3=0}^{u+v-1} \BK{1-\frac{j_3}{2l}} \right..
\end{align*}
For each $u$ and $v$, $\lim_{l\to\infty} F^{u,v}_l(a)=a^u(1-a)^v$
\emph{uniformly} for $a\in[0,1]$.
Thus for any $l'>0$,
\begin{align*}
	\varlimsup_{l\to\infty} \max_{|i|\le l} 
	\bbE_\nu \BK{  V_{l'} \middle| \eta^l(0) = i/l }
\leq
	\sup_{a\in[0,1]}
	\bbE_{\nu_a} \Bigg(\Bigg| \avsum_{|x|\le l'} \tau_x \Psi - \bbE_{\nu_a}(\Psi) \Bigg|\Bigg).
\end{align*}
Finally, since $\nu_a$ is the product of i.i.d.\ measures and $\Psi$ is cylinder,
\begin{align*}
	\lim_{l'\to\infty}\sup_{a\in[0,1]} \bbE_{\nu_a} 
	\Bigg[\Big( \avsum_{|x|\le l'} \tau_x \Psi - \bbE_{\nu_a}(\Psi) \Big)^2\Bigg]
=
	0,
\end{align*}
concluding \eqref{eq:prop:oneBK:2}.
\end{proof}

\begin{proof}[Proof of Proposition \ref{prop:twoBK}]
By \eqref{eq:double:av}, 
we have 
\begin{align*}
	\eta^{\deltaE}(x) = \avsum_{|x''|\le\deltaE} \eta^l(x+x'') + O(l(\deltaE)^{-1}).
\end{align*}
For each $x'$, 
the contribution of those $x''$ with $|x'-x''|\le 2l$
to $\BK{\avvsum_{|x''|\le\deltaE} \eta^l(x+x'')}$ is of $O(l(\deltaE)^{-1})$.
Thus we reduce \eqref{eq:two:block} to showing
\begin{align}\label{eq:prop:twoBK:1}
	\varlimsup_{l\to\infty} \varlimsup_{\delta\to 0}\varlimsup_{\E\to 0}
	\sup_{ 2l<|j|\le \delta \E^{-\frac12} }
	\bbE_\nu \Bigg[
		\avsum_{|x|\le \E^{-1}R} 
		\fNTE \absBK{ \eta^l(x+j) - \eta^{l}(x) } 
	\Bigg]
	=0.
\end{align}
Let $\calG^j_l$ be the $\sigma$-algebra corresponding to
the restriction to $\Lambda_l\cup (j+\Lambda_l)$, and let
\begin{align}
	g^{l,j} := \bbE_\nu[g|\sigma(\calG_l,\calG^j_l)].
\end{align}
Similar to \eqref{eq:prop:oneBK:fl}, we rewrite the expectation of \eqref{eq:prop:twoBK:1} as
$
	\bbE_\nu \sqBK{ (\bar f_{T,\E})^{l,j} \absBK{ \eta^l(j) - \eta^{l}(0) } }.
$
Since $|j|>2l$, $(\bar f_{T,\E})^{l,j}$ is a probability density function on
the configuration of two disjoint intervals $\Lambda_l\cup(j+\Lambda_l)$.
By translating the interval $(j+\Lambda_l)$ to $(2l+\Lambda_l)$,
we obtain a probability density function $(\bar \fkf_{T,\E})^{l,j}$ on
$\calD'_l:=\curBK{\pm}^{\Lambda_l\cup(2l+\Lambda_l)}$.
We further write \eqref{eq:prop:twoBK:1} as
$
	\bbE_\nu \sqBK{ (\bar \fkf_{T,\E})^{l,j} \absBK{ \eta^l(l) - \eta^{l}(0) } }.
$
Similar to \eqref{eq:prop:oneBK:1}, 
by taking limits in
$
	\bbE_\nu \sqBK{ (\bar \fkf_{T,\E})^{l,j} \absBK{ \eta^l(l) - \eta^{l}(0) } }
$
we reduce \eqref{eq:prop:twoBK:1} to showing
\begin{align}\label{eq:prop:twoBK:2}
	\varlimsup_{l\to\infty}
	\bbE_\nu \BK{ \fkf^{l}_\infty\absBK{ \eta^l(l) - \eta^{l}(0) } } = 0,
\end{align}
where $\fkf^{l}_\infty$ is the limiting probability density function on $\calD'_l$
\begin{align*}
	\fkf^{l}_\infty := 
	\varlimsup_{\delta\to 0} \varlimsup_{\E\to 0} \sup_{2l<|j|\le\deltaE} (\bar \fkf_{T,\E})^{l,j}.
\end{align*}

For any probability density function $g$ on $\calD'_l$, define 
\begin{align*}
	\tilD^\infty_l(g) &:= \Dirilinf(g) + \Dirilinf(\tau_{2l}g) + \tilDelta(g),
\\
	\tilDelta(g) &:= r_1 
	\bbE_\nu\sqBK{ \BK{ (\sigma_{\frac12,\frac12+2l}g)^\frac12 - g^\frac12 }^2 }.
\end{align*}
The Dirichlet forms $\Dirilinf\dotBK$ and $\Dirilinf(\tau_{2l}\,\cdot\,)$ correspond to 
hops within $\Lambda_l$ and within $2l+\Lambda_l$, respectively,
and $\tilDelta$ corresponds to hops between $\frac12$ and $\frac12+2l$.
Similar to \eqref{eq:prop:oneBK:Diri:conv} and \eqref{eq:prop:oneBK:Dirichlet:bd},
we have $\Dirilinf(\fkf^{l}_\infty)=0$ and $\Dirilinf(\tau_{2l}\fkf^{l}_\infty)=0$.
As for $\tilDelta$, we have
\begin{align}\label{eq:prop:twoBK:tilDelta}
	\tilDelta((\bar \fkf_\TE)^{l,j})
=
	r_1 
	\bbE_\nu\sqBK{ \BK{ (\sigma_{\frac12,\frac12+j} (\barf_{T,\E})^{l,j})^\frac12 - ((\barf_{T,\E})^{l,j})^\frac12 }^2 }.
\end{align}
Without lost of generality, assume $j>0$.
Since for $y_1<y_2$ the swap $\sigma_{y_1,y_2}$ can be decomposed as
\begin{align*}
	\sigma_{y_1,y_2} g = \sigma_{y_1,y_1+1} \cdots \sigma_{y_2,y_1+1} \cdots \sigma_{y_1+2,y_1+1} \sigma_{y_1+1,y_1} g,
\end{align*}
we have
\begin{align*}
	\sigma_{\frac12,\frac12+j} g - g
=
	\sum_{i=1}^{j} \BK{ \sigma_{\bari,\bari+1} g_i - g_i }
	+ \sum_{i=1}^{j-1} \BK{ \sigma_{j-\bari,j-\bari-1} \tilg_i - \tilg_i },
\end{align*}
where each $g_i$ and $\tilg_i$ is a swapped $g$ (i.e.\ $\sigma_{y_1,y_2}\sigma_{y'_1,y'_2}\cdots g$).
Furthermore, since $\nu$ is invariant under swapping, $g$, $g_i$, and $\tilg$ are equal in distribution under $\nu$.
Hence, applying the Cauchy--Schwartz inequality to \eqref{eq:prop:twoBK:tilDelta} and using \eqref{eq:Diri:infty:Dyy:variant}, 
we obtain
\begin{align*}
	\tilDelta\BK{ (\bar \fkf_\TE)^{l,j} }
\leq
	C j \sum_{i=1}^j D^\infty_{\bari,\bari+1} \BK{ (\barf_\TE)^{l,j} }.
\end{align*}
By using the convexity of $D^\infty_{\bari,\bari+1}$ similar to \eqref{eq:prop:oneBK:Diri:conv},
we further get
\begin{align*}
	\tilDelta\BK{ (\bar \fkf_\TE)^{l,j} }
\leq
	C j^2 \E \DE(\fNTE) \le C j^2 \E. 
\end{align*}
Since $|j|\le \deltaE$, using the lower-semicontinuity of $\tilDelta$, 
we conclude $\tilDelta(\fkf^{l}_\infty)=0$, yielding $\tilD^\infty_l(\fkf^{l}_\infty)=0$.
 
From $\tilD^\infty_l(\fkf^{l}_\infty)=0$ we deduce that $\fkf^{l}_\infty$ 
is a function of the total number of particles in $\Lambda_l\cup(2l+\Lambda_l)$. 
Hence similar to \eqref{eq:prop:oneBK:2} 
we can decomposed the expectation of \eqref{eq:prop:twoBK:2} according to
$\sum_{y\in\Lambda_l\cup(2l+\Lambda_l)} \eta(y)$, 
and use the same argument following \eqref{eq:prop:oneBK:2}
to conclude \eqref{eq:prop:twoBK:2}.
\end{proof}

\begin{proof}[Proof of Lemma \ref{lem:pre:replace}]
Given $T>0$, $n=1,2$, $\phi\in C^0_{[-R,R]}(\bbR)$,
and $\Phi_t(x)$ as defined in \eqref{eq:Phi:defn},
it suffices to show that $\lim_{\E\to 0} \bbE(|\UE|)=0$, where
\begin{align*}
	\UE :=
	\fint_0^{T\E^{-2}} 
	\avsum_{|x|\le\E^{-1}R} \Phi_t(x) \, 
	\ZE_s(x)^n \, \phi(\E x) ds.
\end{align*}
For the Lipschitz cylinder function $\Psi(\eta)=\prod_{i=1}^{n_0}(\eta)$
we have $\Phi_t(x)=(\tau_x\Psi)(\eta_t)$.
Clearly, $\tilPsi(a)=a^{n_0}$.

Using \eqref{eq:double:av} and $\VertBK{\phi}_\infty<\infty$,
we obtain $|\UE| \le C( |\UEd_1| + |\UEd_2|)$, where
\begin{align}
	\label{eq:lem:pre:replace:U1}
	\UEd_1 &:=
	\fint_0^{T\E^{-2}} 
	\avsum_{|x|\le\E^{-1}R} 
	\sqBK{
		 \quad
		 \avsum_{0\le x' \le 2\delta \E^{-1/2}}
		 (\tau_{x+x'}\Psi)(\eta_s) \Big( \ZE_s(x+x')^n \phi(\E(x+x')) \Big) 
		}ds,
\\
	\notag
	\UEd_2 &:=
	\delta^{-1}\E^\frac12 
	\avsum_{|x\pm \E^{-1}R|\le \deltaE} 
	\fint_0^{T\E^{-2}} \ZE_s(x)^n ds.
\end{align}
By \eqref{eq:moment:esti}, $\lim_{\E\to 0}\bbE(|\UEd_2|)=0$,
for any $\delta>0$.

The proof is completed upon showing
$
	\displaystyle
	\varlimsup_{\delta\to 0}
	\varlimsup_{\E\to 0}\bbE(|\UEd_1|)=0
$,
which we next do.
Add and subtract $Z_s(x)^n\phi(\E x)$ inside the parentheses of \eqref{eq:lem:pre:replace:U1},
add and subtract 
\begin{align*}
	\BK{ \eta^{\delta\E^{-1/2}}(x+\delta\E^{-1/2}) }^{n_0}
	\ZE_s(x)^n  \phi(\E x)
=
	\tilPsi \BK{ \eta^{\delta\E^{-1/2}}(x+\delta\E^{-1/2}) }
	\ZE_s(x)^n  \phi(\E x)
\end{align*}
inside the bracket of \eqref{eq:lem:pre:replace:U1},
and make the change of variable $x+\delta\E^{-\frac12} \mapsto x$.
We then get 
$
	|\UEd_1|\le \UEd_{11}+\UEd_{12}+\UEd_{13}
$,
where
\begin{align*}
	&
	\UEd_{ij} :=
	\fint_0^{T\E^{-2}}  
	\avsum_{|x-\delta\E^{-\frac12}|\le\E^{-1}R} (\tilUEd_{ij})_s(x) ds,
\\
	&
	(\tilUEd_{11})_s(x) 
	:= \VertBK{\Psi}_\infty
	\avsum_{|x'| \le \delta \E^{-1/2}}
	\Big| 
		\ZE_s(x+x')^n \phi(\E(x+x'))
\\
		&
		\hphantom{(\tilUEd_{11})_s(x) := \VertBK{\Psi}_\infty\avsum_{|x'| \le \delta \E^{-1/2}}}
	 	- \ZE_s(x-\deltaE)^n \phi(\E (x-\deltaE)) 
	\Big|,
\\
	&
	(\tilUEd_{12})_s(x) :=
	(\tau_x V_{\delta\E^{-1/2}})(\eta_s) \
	\ZE_s(x-\delta\E^{-\frac12})^n \VertBK{\phi}_\infty,
\\
	&
	(\tilUEd_{13})_s(x) :=
	\Big| \eta^{\delta\E^{-\frac12}}_s(x) \Big|^{n_0} 
	\ZE_s(x-\delta\E^{-\frac12})^n \VertBK{\phi}_\infty.
\end{align*}
By \eqref{eq:Holder:esti:space} and the continuity of $\phi$, 
$
	\lim_{\E\to 0} \bbE\BK{\UEd_{11}} =0,
$
for any $\delta>0$.
Applying the Cauchy--Schwartz inequality to $\UEd_{12}$, we obtain
\begin{align}\label{eq:lem:pre:replace:CS:ineq}
\begin{split}
	[\bbE(\UEd_{12})]^2
	&\leq
	\bbE \Bigg( 
		\fint_0^{T\E^{-2}} \avsum_{|x-\delta\E^{-\frac12}|\le\E^{-1}R}
		(\tau_x V_{\delta\E^{-1/2}})^2(\eta_s) ds
		\Bigg)
\\
	&
	\quad\quad\quad
	\bbE \Bigg( 
		\fint_0^{T\E^{-2}} 
		\avsum_{|x-\delta\E^{-\frac12}|\le\E^{-1}R}
		\ZE_s(x-\delta\E^{-\frac12})^{2n} ds
		\Bigg).
\end{split}
\end{align}
By \eqref{eq:V:defn}, $V_n$ is bounded (in $n$) for any bounded $\Psi$,
so we can replace $(\tau_x V_{\delta\E^{-1/2}})^2(\eta_s)$ in \eqref{eq:lem:pre:replace:CS:ineq}
by $(\tau_x V_{\delta\E^{-1/2}})(\eta_s)$.
By Lemma \ref{lem:replace}, 
the first expectation of \eqref{eq:lem:pre:replace:CS:ineq} goes to zero 
under the iterated limit $(\varlimsup_{\delta\to 0} \varlimsup_{\E\to 0})$,
and by \eqref{eq:Z:moment}
the second expectation of \eqref{eq:lem:pre:replace:CS:ineq} is bounded in $\E$ and $\delta$,
yielding
$
	\displaystyle
	\varlimsup_{\delta\to 0} \varlimsup_{\E\to 0}
	\bbE\BK{\UE_{12}} =0.
$
Finally, for $U_{13}$,
given any $|x-\delta\E^{-1/2}|\le\E^{-1}R$ and $s\in[0,\E^{-2}T]$, define
\begin{align*}
	\calD_1(s,x) &:= 
	\curBK{ 
		|\ZE_s(x-\delta\E^{-\frac12})-\ZE_s(x+\delta\E^{-\frac12})| 
		\le \E^\frac18 
		},
\\
	\calD_2(s,x) &:= 
	\curBK{ \ZE_s(x-\delta\E^{-\frac12}) > \E^\frac{1}{16} }.
\end{align*}
By \eqref{eq:Holder:esti:space} for $u=\frac12$ and the Markov inequality
we have
$
	\bbP \BK{\calD_1(s,x)}
	\leq C \E^{\frac14-\frac18}.
$
Applying the Cauchy--Schwartz inequality and \eqref{eq:moment:esti}, we get
\begin{align}\label{eq:prop:replace:U13:1}
	\bbE[(\tilUEd_{13})_s(x) \bbbone_{\calD_1(s,x)^c}]
\leq
	\bbP(\calD_1(x,s)^c)^\frac12 \VertBK{\ZE_s(x)^n}_2
\leq
	C \E^\frac{1}{16}.
\end{align}
Indeed, for $\calD_2(s,x)$ we have
\begin{align}\label{eq:prop:replace:U13:2}
	\bbE[(\tilUEd_{13})_s(x) \bbbone_{\calD_2(s,x)^c}]
\leq
	C \E^\frac{1}{16}.
\end{align}
From the definitions \eqref{eq:Z:defn} and \eqref{eq:etal:defn}, we have
\begin{align*}
	\eta^{\delta\E^{-\frac12}}_s(x)
=
	-\frac{1}{2\delta\lambda} 
	\log
	\frac{ \ZE_s(x+\delta\E^{-\frac12}) }{ \ZE_s(x-\delta\E^{-\frac12}) }.
\end{align*}
Since on $\calD_1(s,x)\cap \calD_2(s,x)$ we have
\begin{align*}
	\frac{\absBK{\ZE_s(x-\delta\E^{-\frac12})-\ZE_s(x+\delta\E^{-\frac12})}}
	{\ZE_s(x-\delta\E^{-\frac12})}
\leq
	\E^\frac{1}{16},
\end{align*}
applying the inequality $|\log(1+t)| \le 2|t|$, which holds for $|t|<\frac12$,
we bound $(\tilUEd_{13})_s(x)\bbbone_{\calD_2(s,x)\cap \calD_3(s,x)}$ by
$
	C(\delta) \E^\frac{n_0}{16} \ZE_s(x-\delta\E^{-\frac12})^n.
$
Further applying \eqref{eq:moment:esti}, we get 
\begin{align}\label{eq:prop:replace:U13:3}
	\bbE\BK{(\tilUEd_{13})_s(x)\bbbone_{\calD_2(s,x)\cap \calD_3(s,x)}} 
\le 
	C(\delta) \E^\frac{n_0}{16}
\leq
	C(\delta) \E^\frac{1}{16}.
\end{align}
Combining \eqref{eq:prop:replace:U13:1}, \eqref{eq:prop:replace:U13:2},
and \eqref{eq:prop:replace:U13:3},
we obtain $ \lim_{\E\to 0} \bbE(\UEd_{13})=0$ for any $\delta>0$,
concluding the proof.
\end{proof}

\section{Exact statistics.}
\label{sec:exact:stats}

In this section we consider the step initial condition \eqref{eq:step:ic}.
Instead of \eqref{eq:calZE:defn}, we use the following modified scaled field
\begin{align}
	\label{eq:calZ:step:ic}
	\scrZE_T(X) 
	&:= \beta'\lambda (2\E^\frac12)^{-1} \calZE_{T}(X),
\end{align}
where $\beta$ and $\beta'$ are as in \eqref{eq:beta:beta'}.
The extract factor of $\beta'\lambda (2\E^\frac12)^{-1}$
ensures that $\scrZE_0\dotBK$ converges to $\delta_0\dotBK$.
Since $\FE_T(X)$, as defined in \eqref{eq:F:engy:defn}, satisfies 
$P_T(X) e^{\FE_T(X)}=\scrZE_T(X) $
and $P_T(X) e^{\scrF(T,X)}= Z_T(X)$,
to prove Theorem \ref{thm:step:stats} it suffices 
to show the convergence of $\calZE_\cdot\dotBK$ to 
the solution of \eqref{eq:SHE:mild} for the initial condition $\delta_0\dotBK$.
However, Theorem \ref{thm:main} does not apply directly 
to the initial condition \eqref{eq:step:ic},
as it violates \eqref{eq:initial:cond:bd} and \eqref{eq:initial:cond:Holder}.
We circumvent this problem following the same argument of \cite{amir11}.
First we show the following holds

\begin{lemma}\label{lem:exact:stats:ic}
For the step initial condition \eqref{eq:step:ic},
given any $j\in\bbN$, $\delta>0$, $u\in(0,1)$, we have
\begin{align}
	\label{eq:exact:stats:moment:esti}
	\VertBK{ \E^{-\frac12} \ZE_{\E^{-2}\delta}(x) }_{2j} &\leq C(j,\delta),
\\
	\label{eq:exact:stats:Holder:esti:space}
	\VertBK{ \E^{-\frac12} \BK{\ZE_{\E^{-2}\delta}(x)-\ZE_{\E^{-2}\delta}(x')} }_{2j} &\leq
	(\E|x-x'|)^\frac{u}{2} C(j,\delta,u).
\end{align}
\end{lemma}

\noindent
From Lemma \ref{lem:exact:stats:ic},
conditions \eqref{eq:initial:cond:bd} and \eqref{eq:initial:cond:Holder}
hold for $\scrZE_{\delta}$, for any $\delta>0$.
We then apply Theorem \ref{thm:main} to conclude 
$\scrZE_\cdot\dotBK \Rightarrow Z_\cdot\dotBK$ on $[\delta,\infty)\times\bbR$,
where $Z_T(X)$ satisfying the \ac{SHE} \eqref{eq:SHE:mild} on $[\delta,\infty)\times\bbR$.
Next, the extension argument in \cite[Section 3]{amir11} 
extends $Z_T(X)$ to $(T,X)\in(0,\infty)\times\bbR$, yielding
$\scrZE_\cdot\dotBK \Rightarrow Z_\cdot\dotBK$ on $(0,\infty)\times\bbR$, and 
\begin{align}\label{eq:exact:stats:SHE:mild}
	Z_{T}(X)
=
	\int_\bbR P_{T-\delta}(X-X') Z_{\delta}(X') dX'
	+
	\int_{\delta}^{T} 
	\int_\bbR P_{T-S}(X-X') Z_{S}(X') W(dX'dS),
\end{align}
for any $\delta>0$.
The proof is then completed upon showing the following

\begin{lemma}\label{prop:exact:stats:initial:conv}
When $\delta\to 0$, 
the RHS of \eqref{eq:exact:stats:SHE:mild} converges weakly to
\begin{align}\label{eq:exact:stats:SHE:delta}
	 P_{T}(X)
	+
	\int_{0}^{T} 
	\int_\bbR P_{T-S}(X-X') Z_{S}(X') W(dX'dS).
\end{align}
\end{lemma}

Now we prove Lemma \ref{lem:exact:stats:ic} and \ref{prop:exact:stats:initial:conv}.

\begin{proof}[Proof of Lemma \ref{lem:exact:stats:ic}]
Let $I_1$, $I_2$, $I_3$, $I_4$ 
denote the first, second, third, fourth terms on the RHS of \eqref{eq:discr:SHE:int}, respectively,
and let $J_i:=\E^{-\frac12} I_i$, $i=1,\ldots,4$.
Note that $J_1$ is deterministic since $\ZE_0$ is.
By \eqref{eq:prop:Holder:4inequ} we have
\begin{align*}
	\VertBK{(\E^{-\frac12}\ZE_t)^2}_{j}
\leq
	4
	\BK{ 
		 (J_1)^2 + \VertBK{(J_2)^2}_{j}
		+\VertBK{(J_3)^2}_{j} + \VertBK{(J_4)^2}_{j}
		}.
\end{align*}
Since $\scrZE_0\dotBK$ approximates the delta function, we have
\begin{align}\label{eq:lem:exact:stats:ic:deltasum}
	\E^\frac12 \sum_x \ZE_0(x) \le C.
\end{align}
Since $J_1=\bfp_t*\E^{-\frac12} \ZE_0$,
using \eqref{eq:p:esti:sum} we have
$
	J_1
	\leq C (\E^2t)^{-\frac12}.
$
From \eqref{eq:prop:Holder:I2:Vert}, \eqref{eq:prop:Holder:I3:Vert}, and \eqref{eq:prop:Holder:I4:Vert},
we obtain
\begin{align}\label{eq:lem:exact:stats:ic:I234}
\begin{split}
	&
	\VertBK{ (J_2)^2 }_j \leq
	C
	\int_0^ t \E (t-s)^{-\frac12}\bfp_\ceilBK{t-s}*\VertBK{(\E^{-\frac12}\ZE_s)^2}_j ds,
\\
	&
	\VertBK{ (J_3)^2 }_j \leq
	 C
	\int_0^t \E^2 \bfp_{t-s} * \VertBK{(\E^{-\frac12}\ZE_s)^2}_j ds,
\\
	&
	\VertBK{ (J_4)^2 }_j \leq
	 C
	\int_0^t \E \sum_{|k|\le m} |\nabla_k\bfp_{t-s}| * \VertBK{(\E^{-\frac12}\ZE_s)^2}_j ds.
\end{split}
\end{align}
Combining the preceding estimates of 
$\Vert(J_i)^2\Vert_{j}$, $i=1,\ldots,4$,
we arrive at the following inequality
\begin{align*}
	\VertBK{(\E^{-\frac12}\ZE_t)^2}_{j} 
\le 
	C \gE_s
	+ 
	C \int_0^t \fE_{t-s} * \VertBK{ (\E^{-\frac12}\ZE_s)^2 }_{j} ds,
\end{align*}
where 
\begin{align*}
	 \gE_t :=
	(\E^{2}t)^{-\frac12} \BK{ \bfp_t * \E^{-\frac12}\ZE_0 },
\quad
	\fE_t := 
	\E t^{-\frac12} \bfp_\ceilBK{t} + \E^2 \bfp_t 
		+ \sum_{|k|\le m} \E |\nabla_k\bfp_t|. 
\end{align*}
After iteration we get
\begin{align}\label{eq:lem:exact:stats:ic:iterate}
	\VertBK{(\E^{-\frac12}\ZE_t)^2}_{j}  
\leq  
	C g_t 
	+
	\sum_{n=1}^\infty \frac{C^n}{n!} 
	\int_{\Delta_n(t)} \fE_{s_1}*\cdots *\fE_{s_n}*\gE_{s_{n+1}} ds_1 \cdots ds_n,
\end{align}
where $\Delta_n(t):=\curBK{(s_1,\ldots,s_{n+1}):s_1+\ldots+s_{n+1}=t}$.
Using \eqref{eq:lem:exact:stats:ic:deltasum},
\eqref{eq:p:esti:Linfty}, \eqref{eq:del:p:esti:Linfty},
\eqref{eq:p:esti:sum}, and \eqref{eq:del:p:esti:sum},
we have 
\begin{align}
	&
	\label{eq:lem:exact:stats:ic:fgSup}
	\gE_t \le C (\E^2 t)^{-1}, \quad \fE_t \le C ( \E t^{-1} + \E^2 t^{-\frac12} ),
\\
	&
	\label{eq:lem:exact:stats:ic:fgSum}
	\sum_x \gE_t(x) \leq C \E^{-2} t^{-\frac12}, \quad
	\sum_x \fE_t(x) \le C \BK{ (\E^2 t)^{-\frac12} + \E^2 }.
\end{align}
For $(s_1,\ldots,s_{n+1})\in\Delta_n(t)$,
we must have $s_{i_*}\ge t(n+1)^{-1}$ for some $i_*\in\curBK{1,\ldots,n+1}$.
Hence, using \eqref{eq:lem:exact:stats:ic:fgSup} and \eqref{eq:lem:exact:stats:ic:fgSum},
we bound the integrand in \eqref{eq:lem:exact:stats:ic:iterate} by
\begin{align*}
	&
	\sum_{i_*=1}^{n}
	C^n \Bigg( \prod_{\substack{i=1,\ldots,n \\ i\neq i_*}} (\E s_i^{-\frac12} + \E^2) \Bigg) 
	\BK{ \E^{-2} (s_{n+1})^{-\frac12} }
	\BK{ \E t^{-1}(n+1) + \E^2 t^{-\frac12} (n+1)^\frac12 }
\\
	&\quad\quad
	+
	C^n \Bigg( \prod_{i=1}^n (\E s_i^{-\frac12} + \E^2) \Bigg) 
	 (\E^2 t(n+1)^{-1})^{-1}.
\end{align*}
For $t\in[\E^{-2}\delta/2,\E^{-2}\tilT]$, integrating in time,
we bound the integral in \eqref{eq:lem:exact:stats:ic:iterate} by $n^2 C(\delta)^n$.
Hence
\begin{align}\label{eq:lem:exact:stats:ic:mom:sup}
	\VertBK{(\E^{-\frac12}\ZE_t)^2}_{j}  \le C(\delta),
	\text{ when } t\geq \frac{\delta\E^{-2}}{2}.
\end{align}
This proves \eqref{eq:exact:stats:moment:esti}.
Similarly, summing over $x$ in \eqref{eq:lem:exact:stats:ic:iterate}
by using \eqref{eq:lem:exact:stats:ic:fgSum},
we then get
\begin{align}\label{eq:lem:exact:stats:ic:mom:sum}
	\sum_x \VertBK{(\E^{-\frac12}\ZE_t)(x)^2}_{j}
	=
	\sum_x \VertBK{(\E^{-\frac12}\ZE_t)(x)}^2_{2j}  
	\le 
	C ( \E^{-2} t^{-\frac12} + 1 ).
\end{align}

Next we prove \eqref{eq:exact:stats:Holder:esti:space}.
Put $n=x'-x$ so that $\ZE_t(x')-\ZE_t(x)=\nabla_n\ZE_t(x)$.
First we have
\begin{align*}
	\VertBK{ \E^{-\frac12} \nabla_n\ZE_{\E^{-2}\delta} }_{2j}
\leq
	\sum_{i=1}^4 \VertBK{  (\nabla_n J_i)_{\E^{-2}\delta} }_{2j}.
\end{align*}
For $i=1$ we have
$
	\VertBK{ (\nabla_n J_1)_{\E^{-2}\delta} }_{2j}
	\leq 
	|\E^{-1}\nabla_n\bfp_{\E^{-2}\delta}| * \E^\frac12 \ZE_0.
$
Applying \eqref{eq:lem:exact:stats:ic:deltasum} and \eqref{eq:del:p:esti:Linfty} 
for $v=\frac{u}{2}$ we get 
$ |\nabla_n\bfp_{\E^{-2}\delta}| * \ZE_0 \le C(\delta) |\E n|^\frac{u}{2}$.
Next, by \eqref{eq:prop:Holder:I2:del:Vert}, \eqref{eq:prop:Holder:I3:del}, 
and \eqref{eq:prop:Holder:I4:del}, we have 
\begin{align}\label{eq:lem:exact:stats:ic:three}
\begin{split}
	&
	\Vert \nabla_n (J_2)_{\E^{-2}\delta} \Vert_{2j}^2
	\leq 
	C \E \int_0^{\E^{-2}\delta} 
	|n|^u (t-s)^{-\frac{1+u}{2}}
	\BK{ \bfp_\ceilBK{t-s} + \tau_n\bfp_\ceilBK{t-s}} * \VertBK{\E^{-\frac12}\ZE_s}^2_{2j} ds.
\\
	&
	\VertBK{ \nabla_n (J_3)_{\E^{-2}\delta} }^2_{2j} \leq
	C |n\E|^u
	\int_0^{\E^{-2}\delta} \E^2 (\bfp_{t-s} + \tau_n\bfp_{t-s}) *
	\VertBK{\E^{-\frac12}\ZE_s}^2_{2j} ds,
\\
	&
	\VertBK{ \nabla_n (J_4)_{\E^{-2}\delta} }^2_{2j} \leq
	C |n\E|^u
	\sum_{|k|\le m} 
	\int_0^{\E^{-2}\delta} \E (|\nabla_k\bfp_{t-s}| + |\tau_n\nabla_k\bfp_{t-s}| ) *
	\VertBK{\E^{-\frac12}\ZE_s}^2_{2j} ds.
\end{split}
\end{align}
To bound the terms in \eqref{eq:lem:exact:stats:ic:three}, we divide the time integrals into
an integral over $(0,\delta\E^{-2}/2)$ and an integral over $(\delta\E^{-2}/2,\delta\E^{-2})$.
For the integral over $(0,\delta\E^{-2}/2)$ apply the inequality 
\begin{align}\label{eq:lem:exact:stats:ic:ab}
	\sum_{x'} |a(x') b(x')| \leq \sup_{x'} |a(x')| \sum_{x'} |b(x')|
\end{align}
with $b=\VertBK{\E^{-\frac12}\ZE_s}^2_{2j}$ and $a$ being the rest of the integrand.
By using \eqref{eq:p:esti:sup}, \eqref{eq:del:p:esti:Linfty}, and \eqref{eq:lem:exact:stats:ic:mom:sum},
we bound the integrals by $C(\delta)|n\E|^{u/2}$.
Similarly, for the integral over $(\delta\E^{-2}/2,\delta\E^{-2})$ apply \eqref{eq:lem:exact:stats:ic:ab}
with $a=\VertBK{\E^{-\frac12}\ZE_s}^2_{2j}$ and $b$ being the rest of the integrand.
By using \eqref{eq:p:esti:sum}, \eqref{eq:del:p:esti:sum}, and \eqref{eq:lem:exact:stats:ic:mom:sup},
we bound the integrals by $C(\delta)|n\E|^{u/2}$.
Hence $\VertBK{ \E^{-\frac12} \nabla_n (I_i)_{\E^{-2}\delta} }^2_{2j} \le C(\delta)|n\E|^{u/2}$, for $i=2,3,4$,
concluding \eqref{eq:exact:stats:Holder:esti:space}.
\end{proof}

\begin{proof}[Proof of Lemma \ref{prop:exact:stats:initial:conv}]
Let $\calA_\delta$, $\calB_\delta$ denote the first and second terms on the RHS 
of \eqref{eq:exact:stats:SHE:mild}, respectively,
and let $\calA$, $\calB$ denote the first and second term of \eqref{eq:exact:stats:SHE:delta},
respectively.
By the It\^o isometry,
\begin{align*}
	\bbE \BK{ \int_0^\delta \int_\bbR P_{T-S}(X-X') Z_S(X') W(dX dS)  }^2
	=
	\int_0^\delta \int_\bbR P_{T-S}(X-X')^2 \bbE\BK{Z_S(X')}^2 dX'dS.
\end{align*}
Using the boundedness of $ P_{T-S}(X-X')^2$ and \eqref{eq:lem:exact:stats:ic:mom:sum}, 
we further bound this expression by 
\begin{align}\label{eq:prop:exact:stats:initial:conv:1}
	C \varlimsup_{\E\to 0} \int_0^\delta \E \sum_{x'}  \bbE\BK{\E^{-\frac12}\ZE_{\beta\E^{-2}S}(\beta'x')}^2 dS
\leq
	C \varlimsup_{\E\to 0} \int_0^\delta \E S^{-\frac12} dS
=
	C \delta^\frac12.
\end{align}
Since the RHS of \eqref{eq:prop:exact:stats:initial:conv:1} converges to zero as $\delta\to 0$,
$\calB_\delta$ weakly converges to $\calB$.

Let $(\KE_i)_T(X):=\lambda\beta 2^{-1}(J_i)_{\beta\E^{-2}T}(\beta'\E^{-1}X)$.
By definition $\calA_\delta$ is the weak limit of
\begin{align*}
	 \sum_{i=1}^4 \lim_{\E\to 0} P_{T-\delta}*(\KE_i)_\delta.
\end{align*}
Since $\calZE_0\dotBK$ approximates the delta function, by \eqref{eq:p:conv:to:P} and \eqref{eq:beta:beta'} we have
$
	(\KE_1)_\delta(X) \xRightarrow[]{\E\to 0} P_\delta(X),
$
yielding 
$
	\lim_{\E\to 0} P_{T-\delta}*(\KE_1)_\delta = P_T.
$
Next, using \eqref{eq:lem:exact:stats:ic:I234}, \eqref{eq:lem:exact:stats:ic:mom:sum}, \eqref{eq:p:esti:sum},
and the boundedness of $P_{T-\delta}$,
for $i=2,3,4$ we have
\begin{align*}
	\lim_{\delta\to 0} \Big\Vert \lim_{\E\to 0} P_{T-\delta}*(\KE_i)_\delta \Big\Vert_2
\leq
	C
	\lim_{\delta\to 0} \lim_{\E\to 0}
	\E \sum_{x'} 
	\VertBK{ (J_i)_{\E^{-2}\delta}(x') }_2
=
	0,
\end{align*}
concluding the proof.
\end{proof}

\appendix
\section{}
\label{sec:app}

\begin{proof}[Proof of Lemma \ref{lem:gamma:tilr}]

We first solve the equation \eqref{eq:matching:eq} without the $O(\E^2)$,
that is
\begin{align}\label{eq:matching:eq:exact}
\begin{split}
	 A^\E\tilrE  
	&= 
	4^{-1} \BK{ u(\E) Ar - \E v(\E)AR\gammaE },
\\
	 B\tilrE 
	&= 
	4^{-1} \BK{ u(\E) R \gammaE - v(\E) r},
\end{split}
\end{align}
where $A$ and $R$ are the $m$-dimensional square matrices
$A_{jk}:=\bbbone_\curBK{j\le k}$ and $R:=\bbbone_\curBK{j=k}r_k$,
and $r:=(r_1,\ldots,r_m)$.
Since $A^\E$ is invertible for all $\E$ small enough,
by multiplying the first equation of \eqref{eq:matching:eq:exact} 
by $(A^\E)^{-1}$ and substituting it into the second equation, 
we arrive at the following equivalent equations
\begin{align}\label{eq:matrix:eq}
\begin{split}
	&
	\BK{ R  + \E \frac{v(\E)}{u(\E)} B(A^\E)^{-1} A R }
	\gammaE 	
	=
	\BK{ \frac{v(\E)}{u(\E)} + B (A^\E)^{-1} A } r,
\\
	&
	\tilrE = 
	 4^{-1}
	\Big( u(\E) (A^\E)^{-1}A r - \E v(\E)(A^{\E})^{-1} A R \gammaE \Big),
\end{split}
\end{align}
Since $R$ has full rank, \eqref{eq:matrix:eq} has a unique solution $(\gamma(\E),\tilr(\E))$.
Moreover, since $(A^\E)^{-1}$, $u(\E)$, $v(\E)$ are $C^\infty$ in $\E$
for small enough $\E$ (even at $\E=0$), and $u(0)=2\lambda>0$,
$\gamma(\E)$ and $\tilr(\E)$ are also $C^\infty$ in $\E$.
By solving \eqref{eq:matrix:eq} at $\E=0$ (notice that $A^\E=\frac{\lambda}{2}A$ at $\E=0$), 
we obtain
\begin{align*}
	\notag
	\gamma_k(0) = \lambda \BK{ \frac{2}{r_k} \sum_{k'=k+1}^m \frac{k'-k}{k} r_{k'} +1 },
\quad
	\tilr_k(0) = r_k.
\end{align*}

By choosing $\tilr^*_k:=\frac{d\tilr_k}{d\E}(0)$,
for any given $\gammaE$ and $\tilrE$ of the form \eqref{eq:gamma:approx} and \eqref{eq:tilr:approx},
we have $\gammaE-\gamma(\E)=O(\E)$ and $\tilrE-\tilr(\E)=O(\E^\frac32)$.
Tracking the coefficients multiplying $\gammaE_k$ and $\tilrE_k$ in \eqref{eq:matching:eq}, 
we find that the difference of $O(\E)$ between $\gammaE$ and $\gamma(\E)$
and the difference of $O(\E^\frac32)$ between $\tilrE$ and $\tilr(\E)$
will only contribute $O(\E^2)$ to \eqref{eq:matching:eq}, concluding the proof.
\end{proof}

Next, we provide some estimates of the semi-discrete heat kernel $\bfp$, 
as defined in \eqref{eq:discr:heat:kernel}.
Solving \eqref{eq:discr:heat:kernel} by Fourier series,
we have 
\begin{align}\label{eq:p:Fourier:expression}
	\bfp_t(x) =
	\frac{1}{2\pi} \int_{-\pi}^\pi e^{-ix\theta} 
	\exp\sqBK{ -t \phi(\theta) } d\theta,
	\quad
	\text{where }
	\phi(\theta): = \sum_{k=1}^m (1-\cos(k\theta))\tilrE_k .
\end{align}
From \eqref{eq:tilr:approx} and \eqref{eq:r:constraint} we have
\begin{equation}\label{eq:gcd:consequence}
	0 < c_0 \leq \phi(\theta)\theta^{-2} \leq C< \infty,
	\text{ for all } \theta \in [-\pi,\pi]\setminus\curBK{0}.
\end{equation}

Indeed,
\begin{align*}
	\E^{-1} \bfp_{\E^{-2}T}(\E^{-1}X) =
	\frac{1}{2\pi} \int_{-\pi\E^{-1}}^{\pi\E^{-1}} 
	e^{-iX\theta} 
	\exp\sqBK{ -T \sum_{k=1}^m \tilrE_k\E^{-2}(1-\cos(\theta\E^{-1})) } d\theta.
\end{align*}
For each $T>0$, by \eqref{eq:gcd:consequence}, the integrand is bounded by $e^{-c_0T\theta^2}$.
Hence by the dominated converge theorem we have, for any $(T,X)\in(0,\infty)\times\bbR$,
\begin{align}\label{eq:p:conv:to:P}
	\lim_{\E\to 0} \E^{-1} \bfp_{\E^{-2}T}(\E^{-1}X) = P_{\alpha T}(X),
\end{align}
where $\alpha$ is defined as in \eqref{eq:alpha:defn}.

\begin{proposition}\label{prop:heat:ker}
Given any $b\ge 0$,
for any $|k|\le m$, $n\in\bbZ$, $v\in[0,1]$, $0\le t<t'<\infty$, and $x\in\bbZ$,
we have
\begin{align}
	\label{eq:p:esti:sup}
	&
	\bfp_{t}(x) \leq e^{C(t'-t)} \bfp_{t'}(x),
\\
	\label{eq:p:holder:time:esti}
	&
	|\bfp_{t'}(x) - \bfp_t(x)|
	\leq 
	 (1\wedge t^{-\frac12-v})(t'-t)^v C,
\\
	\label{eq:del:p:holder:time:esti}
	&
	|\nabla_k\bfp_{t'}(x) - \nabla_k\bfp_t(x)|
	\leq 
	 (1\wedge t^{-1-v})(t'-t)^v C,		
\\
	\label{eq:p:esti}
	&
	\bfp_t(x)
	\leq 
	 C(b) (1\wedge t^{-\frac12}) e^{-b|x|(1\wedge t^{-1/2})},
\\
	\label{eq:del:p:esti}
	&
	|\nabla_n\bfp_t(x)|
	\leq 
	C(b) (1\wedge t^{-\frac{1+v}{2}}) |n|^v e^{-b|x|(1\wedge t^{-1/2})},
\\
	\label{eq:del:p:holder:esti}
	&
	|\nabla_n\nabla_k\bfp_t(x)|
	\leq 
	 C(b) (1\wedge t^{-\frac{2+v}{2}}) |n|^v e^{-b|x|(1\wedge t^{-1/2})}.	
\end{align}
\end{proposition}

\begin{proof}
The estimate \eqref{eq:p:esti:sup} is derived the same way 
as \cite[(A.5)]{bertini97}.

Next we prove prove \eqref{eq:p:holder:time:esti}.
Take the difference 
$\bfp_{t'}(x)-\bfp_t(x)$ using \eqref{eq:p:Fourier:expression},
and then use \eqref{eq:gcd:consequence} and 
the readily verified identity 
\begin{align}\label{eq:prop:heat:ker:readly:verified}
	1-e^{-a} \le 1\wedge a^v
\end{align}
(which holds for $a\geq 0$)
to bound the integrand by 
$
	C [1\wedge((t'-t)\theta^2)^u] e^{-c_0t\theta^2}.
$
This expression can be further bounded by 
\begin{align}\label{eq:prop:heat:ker:1}
	C (t'-t)^u(1\wedge\theta^{2u}) e^{-c_0t\theta^2},
\end{align}
because $|\theta|\leq \pi$.
To get \eqref{eq:p:holder:time:esti},
First use the bound $C(t'-t)^u$ in \eqref{eq:prop:heat:ker:1} and integrate over $\theta$,
and then use the bound $C(t'-t)^u\theta^{2u}e^{-c_0t\theta^2}$ in \eqref{eq:prop:heat:ker:1} and integrate over $\bbR$. 
Similarly, modifying \eqref{eq:p:Fourier:expression} 
by taking the discrete gradient,
we get \eqref{eq:del:p:holder:time:esti} through the same reasoning.

To prove \eqref{eq:p:esti}, \eqref{eq:del:p:esti}, 
and \eqref{eq:del:p:holder:esti}
we derive another integral expression of $\bfp_t$.
First by making the change of variable 
$z=e^{i\theta}$ in \eqref{eq:p:Fourier:expression},
we turn the integral over $\theta \in [-\pi,\pi]$ 
into a contour integral of $\frac{1}{2\pi i}  z^{-x-1} e^{-t\psi(z)}$
along $\curBK{|z|=1}\subset \bbC$,
where $\psi(z):= \sum_k \tilrE_k [1-(z+z^{-1})2^{-1}]$.
Since this integrand is holomorphic on $\bbC\setminus\{0\}$,
we can deform the original contour $\curBK{|z|=1}$
and integrate along $\curBK{|z|=1+\delta}$, where
\begin{align}\label{eq:prop:heat:ker:delta:defn}
	\delta := e^{b\sign(x)(1\wedge t^{-1/2})}-1.
\end{align}
Making another change of variable $(1+\delta)e^{-i\theta} =z$ 
in this new contour integral,
we obtain
\begin{align}\label{eq:p:Fourier:variant}
	\bfp_t(x)
&=
	\frac{1}{2\pi}
	\int_{-\pi}^\pi
	 (1+\delta)^{-x} e^{-ix\theta} 
	 \exp\sqBK{ -t \psi((1+\delta)e^{i\theta}) }
	 d\theta.
\end{align}
Next, by the definition of $\psi$ and $\phi$ we have
\begin{align}\label{eq:psi:esti}
	\psi((1+\delta)e^{i\theta}) 
= 
	\phi(\theta) +
	i\delta \sum_{k=1}^m \tilrE_k \sin(k\theta)
	+
	\frac{\delta^2}{2(1+\delta)}
	\sum_{k=1}^m \tilrE_k e^{-ik\theta}.
\end{align}
Combining \eqref{eq:gcd:consequence},\eqref{eq:p:Fourier:variant},
and \eqref{eq:psi:esti}, 
and integration over $[-\pi,\pi]$, 
we bound $\bfp_t(x)$ by 
$
	C (1+\delta)^{-x} (1\wedge t^{-\frac12}) e^{C\delta^2t}.
$
By \eqref{eq:prop:heat:ker:delta:defn} and \eqref{eq:prop:heat:ker:readly:verified}, 
we have $\delta^2 t \le C(b)$ and 
$
 (1+\delta)^{-x} = e^{-b|x|(1\wedge t^{-1/2})},
$
concluding \eqref{eq:p:esti}.

Next, we turn to \eqref{eq:del:p:esti}.
Modifying \eqref{eq:p:Fourier:variant} by taking the discrete gradient $\nabla_n$,
we get 
\begin{equation}\label{eq:p:Fourier:variant:1}
	\nabla_n\bfp_t(x)
=
	\frac{1}{2\pi} \int_{-\pi}^\pi S(\theta) 
	\exp\sqBK{ -t \BK{1-\psi((1+\delta)e^{i\theta}) } } d\theta,
\end{equation}
where $S(\theta): =(1+\delta)^{-(x+n)} e^{-i(x+n)\theta} - (1+\delta)^{-x} e^{-ix\theta}$.
Write $S(\theta)$ as the sum of 
$(1+\delta)^{-x} e^{-i(x+n)\theta} [(1+\delta)^{-n}-1]$ and 
$(1+\delta)^{-x} e^{-ix\theta} (e^{-i\theta n}-1)$.
By \eqref{eq:prop:heat:ker:delta:defn} and \eqref{eq:prop:heat:ker:readly:verified}, 
we have 
\begin{align*}
	&
	(1+\delta)^{-n}-1
	\leq C(b) |n(1\wedge t^{-1/2})|^v,
\\
	&
	|e^{i\theta n}-1| \le C |\theta n|^v,
\end{align*}
yielding, 
\begin{equation}\label{eq:prop:app:3}
	|S(\theta)| 
	\le C(b) |n|^v ((1\wedge t)^{-v/2} + |\theta|^v) 
	e^{-b|x|(1\wedge t^{-1/2})}.
\end{equation}
Combining this with \eqref{eq:p:Fourier:variant:1},
and \eqref{eq:psi:esti}, 
and integrating over $\theta$, we obtain \eqref{eq:del:p:esti}.

As for \eqref{eq:del:p:esti},
similar to \eqref{eq:del:p:esti}, we modify \eqref{eq:p:Fourier:variant} to get
\begin{equation}\label{eq:p:Fourier:variant:2}
	\nabla_n\nabla_k\bfp_t(x)
=
	\frac{1}{2\pi} \int_{-\pi}^\pi T(\theta) 
	\exp\sqBK{ -t \BK{1-\psi((1+\delta)e^{i\theta}) } } d\theta,
\end{equation}
where $T(\theta) := S(\theta)\tilS(\theta)$,
and $\tilS(\theta):=(1+\delta)^{-k}e^{-ik\theta}-1$.
Write $\tilS(\theta)$ as the sum of 
$[(1+\delta)^{-k}-1]e^{-ik\theta}$ and $(e^{-i\theta k}-1)$.
Since $|k|\le m$, by \eqref{eq:prop:heat:ker:delta:defn} and \eqref{eq:prop:heat:ker:readly:verified},
we have
\begin{align*}
	&
	(1+\delta)^{-k}-1
	\leq C(b)(1\wedge t^{-1/2}),
\\
	&
	|e^{i\theta k}-1| \le C |\theta|.
\end{align*}
Therefore, $|\tilS(\theta)| \le C(b)((1\wedge t^{-1/2}) + |\theta|)$.
Combining this inequality with \eqref{eq:prop:app:3}, 
\eqref{eq:p:Fourier:variant:2}, and \eqref{eq:psi:esti}
and integrating over $\theta$, we obtain \eqref{eq:del:p:holder:esti}.
\end{proof}

Proposition \ref{prop:heat:ker} immediately implies the following corollary
\begin{corollary}
Given any $b\ge 0$,
for any $|k|\le m$, $v\in[0,1]$, $t\in[0,\tilT\E^{-2}]$, $x\in\bbZ$, and $j\in\bbN$,
we have
\begin{align}
	\label{eq:p:esti:sum}
	&
	\sum_x \bfp_t(x) e^{b\E|x|}
	\leq 
	 C(b),
\\
	\label{eq:del:p:esti:sum}
	&
	\sum_x |\nabla_n\bfp_t(x)|
	\leq 
	C(b) t^{-\frac{v}{2}} |n|^v,
\\
	\label{eq:del:del:p:esti:sum}
	&
	\sum_x |\nabla_n\nabla_k\bfp_t(x)|
	\leq 
	C(b) t^{-\frac{1+v}{2}} |n|^v,
\\
	\label{eq:p:esti:Linfty}
	&
	\sup_x \bfp_t(x) e^{b\E|x|}
	\leq 
	 t^{-\frac12} C(b),
\\
	\label{eq:del:p:esti:Linfty}
	&
	\sup_x |\nabla_n\bfp_t(x)| e^{b\E|x|}
	\leq 
	C(b) (1\wedge t^{-\frac{1+v}{2}}) |n|^v.	
\end{align}
\end{corollary}

\bibliographystyle{abbrv}
\bibliography{WAnSEP}

\end{document}